\documentclass{daj}

\usepackage{hyperref}
\usepackage{amsmath}
\usepackage{amssymb}
\usepackage{bbding}
\usepackage{array}
\usepackage{amsthm}
\usepackage{graphicx}

\newtheorem{thm}{Theorem}[section]

\theoremstyle{plain}
\newtheorem{claim}[thm]{Claim}
\newtheorem{cor}[thm]{Corollary}
\newtheorem{prop}[thm]{Proposition}
\newtheorem{lem}[thm]{Lemma}

\theoremstyle{definition}
\newtheorem{defn}[thm]{Definition}
\newtheorem{example}[thm]{Exampls}
\theoremstyle{remark}

\theoremstyle{remark}
\newtheorem{rem}[thm]{Remark}
\theoremstyle{definition}
\newtheorem{problem}[thm]{Problem}
\DeclareMathOperator{\diam}{diam}

\DeclareMathOperator{\safe}{safe}

\dajAUTHORdetails{%
  title = {Irreducibility and periodicity in $\mathbb{Z}^{2}$ symbolic systems}, 
  author = {Michael Hochman},
  plaintextauthor = {Michael Hochman},
  plaintexttitle = {Irreducibility and periodicity in Z2 symbolic systems}, 
  keywords = {keyword, keyword, etc.},
}   

\dajEDITORdetails{%
   year={2025},
   number={17},
   received={18 September 2024},   
   published={12 September 2025},  
   doi={10.19086/da.144095},       
}   

\begin{document}

\begin{frontmatter}[classification=text]


\author[hochman]{Michael Hochman\thanks{Research supported by ISF research grant 3056/21.}}

\begin{abstract}
We show that there exist $\mathbb{Z}^{2}$ symbolic systems that are
strongly irreducible and have no (fully) periodic points.

\end{abstract}
\end{frontmatter}


\section{Introduction}

\subsection{Statement of the main result}

Many mixing conditions have been introduced in the theory of $\mathbb{Z}^{d}$
symbolic dynamics, with the goal of finding a class of systems that
behaves analogously to one-dimensional case. This line of investigation
probably begins with Ruelle's notion of specification \cite{Ruelle1973},
which was arose in his study of the thermodynamic formalism. Other
aspects of the theory have lead to various other notions: block gluing,
corner gluing, uniform filling and square mixing (e.g. \cite{Desai2009,RobinsonSahin2001,Lightwood2003,Pavlov2015,BrecenoMcGoffPavlov2018,GangloffSablik2021}). 

The strongest and perhaps most natural notion is that of strong irreducibility
(e.g. \cite{BurtonSteif1994}): Given\textbf{ $E,F\subseteq\mathbb{Z}^{d}$},
one says that a $\mathbb{Z}^{d}$-subshift $X\subseteq A^{\mathbb{Z}^{d}}$
admits \textbf{gluing} along\textbf{ $E,F$ }if for every $x,y\in X$
there exists a point $z\in X$ such that $z|_{E}=x|_{E}$ and $z|_{F}=y|_{F}$.
In this case $z$ is called a gluing of $x|_{E},y|_{F}$. A subshift\textbf{
$X$ }is \textbf{strongly irreducible (SI) with gap $g\geq0$ }if
it admits gluing along every pair of sets $E,F$ satisfying $d(E,F)>g$\textbf{.}

Strongly irreducible systems behave in many respects like one-dimensional
mixing shifts of finite type, and in $\mathbb{Z}^{2}$ the resemblance
is even stronger. They also exhibit fewer logical complications than
shifts of finite type. We summarize these similarities in the table
below:

\noindent %
\begin{tabular}{|>{\centering}m{4cm}|>{\centering}m{1cm}|>{\centering}m{1cm}|>{\centering}m{3.5cm}|>{\centering}m{3.5cm}|}
\hline 
\hline 
\noalign{\vskip\doublerulesep}
 & $\mathbb{Z}$-SFT & $\mathbb{Z}^{d}$-SFT & Strongly Irreducible

$\mathbb{Z}^{d}$-subshift & Strongly Irreducible

$\mathbb{Z}^{d}$-SFT\tabularnewline[\doublerulesep]
\hline 
\noalign{\vskip\doublerulesep}
Language is 

decidable & \Checkmark{} & \XSolid{}  & Undefined (requires

recursive description) & \Checkmark{} \cite{HochmanMeyerovitch2010}\tabularnewline[\doublerulesep]
\hline 
\noalign{\vskip\doublerulesep}
Mixing implies unique measure of 

maximal entropy  & \Checkmark{} & \XSolid{} & No, but fully 

supported & No, but fully 

supported \cite{BurtonSteif1994}\tabularnewline[\doublerulesep]
\hline 
\noalign{\vskip\doublerulesep}
Total transitivity 

implies mixing & \Checkmark{} & \XSolid{} & Automatic & Automatic\tabularnewline[\doublerulesep]
\hline 
\noalign{\vskip\doublerulesep}
Factor onto lower 

entropy full shifts & \Checkmark{} & \XSolid{} & \Checkmark{} \cite{HuczekKopacz2021} & \Checkmark \cite{Desai2009}\tabularnewline[\doublerulesep]
\hline 
\noalign{\vskip\doublerulesep}
Subshifts that factors 

into $X$ embed in $X$ & \Checkmark{} & \XSolid{} & \Checkmark{} \cite{Bland2022} & \Checkmark{} \cite{Lightwood2003,Lightwood2004}\tabularnewline[\doublerulesep]
\hline 
\noalign{\vskip\doublerulesep}
Periodic points 

exist / are dense  & \Checkmark{} & \XSolid{} & \textbf{THIS PAPER} & \Checkmark{} if $d\leq2$, \cite{Lightwood2003,Lightwood2004}

\textbf{?~} if $d\geq3$~~~~~~~~~~~~~~\tabularnewline[\doublerulesep]
\hline 
\end{tabular}

In the present paper we consider the question, raised in \cite{CeccheriniSilbersteinCoonaers2012},
of existence of periodic points in SI systems (without assuming the
SFT property). Here, a \textbf{periodic point} means a point with
finite orbit under the shift action. For $\mathbb{Z}$-subshifts,
an elegant argument by Bertrand \cite[Theorem 1]{Bertrand1988} and,
independently,  Weiss \cite[Theorem 5.1]{CeccheriniSilbersteinCoonaers2012}
shows that strongly irreducible $\mathbb{Z}$-subshifts admit dense
periodic points (see \cite{CeccheriniSilbersteinCoonaers2012}), and
as we have noted, the same is true in $\mathbb{Z}^{2}$ if we add
an assumption that the subshift is an SFT. 

Joshua Frisch has pointed out to us that the question of existence
of periodic points in SI $\mathbb{Z}^{2}$ subshifts is heuristically
related to the question of their existence in SI SFTs in $\mathbb{Z}^{3}$.
Indeed, if $X$ is a SI $\mathbb{Z}^{2}$-subshift, then for every
$k$, the $\mathbb{Z}$-subshift consisting of restrictions of its
configurations to $\mathbb{Z}\times\{1,\ldots,k\}$ (with $\mathbb{Z}$
acting by horizontal shift) is again SI, so by Weiss's result it contains
$\mathbb{Z}$-periodic points. Similarly, if $Y$ is a SI $\mathbb{Z}^{3}$
SFT, then for every $k$ one can consider locally admissible configurations
on $\mathbb{Z}^{2}\times\{1,\ldots,k\}$, and this is a SI $\mathbb{Z}^{2}$
SFT, so by Lightwood's result it contains $\mathbb{Z}^{2}$-periodic
points. In both cases, there is no obvious way to extend this lower-dimensional
periodicity to periodic points of the original system, but it lends
some support to the idea that periodic points may exist.

Our main result is 
\begin{thm}
\label{thm:main}For every $d\geq2$, there exist strongly irreducible
$\mathbb{Z}^{d}$ subshifts without periodic points.
\end{thm}
We prove the theorem for $d=2$, but this implies it for all $d>2$
by considering the subshifts whose $2$-dimensional slices belong
to our example. Our construction is non-canonical and far from optimal
in every sense. We obtain a result for gap $10$ and a very large
alphabet, and have made no attempt to minimize these parameters. But
we see no obstruction to there existing examples with gap $1$ and
alphabet $\{0,1\}$.

We also note there do not exist so many ``non-trivial'' examples
of SI subshifts -- previous examples either had ``free symbols''
or ``free subsets of symbols'', or were factors of systems with
such symbols. Our example does not rely on this mechanism.

.

\subsection{\label{subsec:Discussion}Discussion}

A point $x\in A^{\mathbb{Z}^{2}}$ is aperiodic if and only if, for
each $n\in\mathbb{N}$, there exists a pair $(u_{n},v_{n})\in\mathbb{Z}^{2}\times\mathbb{Z}^{2}$
with $u_{n}=v_{n}\bmod n$, and $x_{u_{n}}\neq x_{v_{n}}$. Such a
pair will be called an \textbf{$n$-aperiodic pair} \textbf{for }$x$.
The set of configurations that posses an $n$-aperiodic pair for a
given $n$, or for all $n$, is closed under shifts, but not under
taking limits. To ensure that the orbit closure of $x$ contains no
periodic points, one must require that there exists a sequence of
sets $\mathcal{W}_{n}\subseteq\mathbb{Z}^{2}\times\mathbb{Z}^{2}$
and a sequence $R_{n}>0$ such that 
\begin{itemize}
\item \textbf{Aperiodicity: }Every $(u,v)\in\mathcal{W}_{n}$ is an $n$-aperiodic
pair for $x$.
\item \textbf{Syndeticity: }Every closed $R_{n}\times R_{n}$-square contains
$u,v$ for some $(u,v)\in\mathcal{W}_{n}$.
\end{itemize}
If we are given the sequence $R_{n}\rightarrow\infty$, the set $X=X_{(R_{n})}$
of all configurations that admit sets $\mathcal{W}_{n}$ as above
is an aperiodic subshift, and if $R_{n}\rightarrow\infty$ quickly
enough then $X\neq\emptyset$. 

Subshifts of the form $X_{(R_{n})}$ are never SI in dimension $d=1$,
for, by Bertrand/Weiss's result, if $X_{(R_{n})}$ were SI it would
contain periodic points, which it does not. Our main theorem shows
that this argument cannot be applied in higher dimensions, i.e. the
non-existence of periodic points in $X_{(R_{n})}$ does not in itself
rule out the possibility that $X_{(R_{n})}$ is SI. However, for at
least some sequences $R_{n}$, we are able to show that the subshift
$X_{(R_{n})}$ is not SI. 
\begin{example}
\label{example}For $A=\{0,1\}$ there exists a sequence $R_{n}\rightarrow\infty$
such that $X_{(R_{n})}\subseteq\{0,1\}^{\mathbb{Z}^{2}}$ is not SI. 
\end{example}
The construction is given in Section \ref{sec:Example}. It demonstrates
well why controlling the effects of gluing is so difficult: when a
finite region is cut out of a configuration and another pattern glued
in its place, the change potentially can create periodicity at any
one of infinitely many ``scales''. Thus, if one wants to glue in
a way that prevents periodicity, one potentially must satisfy infinitely
many constraints, arising from symbols arbitrarily far away. At the
same time, we are free to re-define only finitely many symbols near
the boundary where the gluing occurred. So there are not always enough
degrees of freedom to ensure aperiodicity.

We note that the example above produced very special sequences $R_{n}$.
We do not know if some sequences $(R_{n})$ can give rise to SI subshifts
$X_{(R_{n})}$. Even if the answer is negative, it seems possible
that subshifts defined as above can lead to SI aperiodic systems in
the following sense. Given a rule $S$ for gluing configurations together,
every subshift $X$ defines a strongly irreducible closure $\left\langle X\right\rangle _{S}=\bigcup_{n=1}^{\infty}X_{n}$,
where $X_{1}=X$ and $X_{n+1}$ is obtained by gluing patterns from
$X_{n}$ in all possible ways using $S$. One can then ask if there
is a sequence $R_{n}>1$ and a gluing rule $S$ such that $\left\langle X_{(R_{n})}\right\rangle _{s}$
is aperiodic?

\subsection{Ideas from the construction}

In order to obtain an aperiodic and SI subshift we will have to strengthen
the aperiodicity and syndeticity assumptions above, thus making the
aperiodicity more ``robust''. 

Our problem is as follows. When $z$ is a gluing of $x|_{E}$ and
$y|_{F}$, it may happen that some pair $(u,v)\in E\times F$ is an
$n$-aperiodic pair for $x$ or for $y$ (i.e. $u=v\bmod n$, and
$x_{u}\neq x_{v}$ or $y_{u}\neq y_{v}$) but not for $z$ (i.e. $z_{u}=x_{u}=y_{v}=z_{v}$). 

Since we are free to re-define the symbols in $(E\cup F)^{c}$, we
can try to resolve this by creating a new aperiodic pair, e.g. by
``moving'' $u$ to the set $(E\cup F)^{c}$, whose size is at least
proportional to the boundary of $E$. But this strategy faces two
problems:
\begin{enumerate}
\item There may not be enough space in $(E\cup F)^{c}$ to create new $n$-aperiodic
pairs for every pair that was destroyed. Indeed, we can only assume
that $(E\cup F)^{c}$ is proportional to the length of the boundary
of $E$, whereas, potentially, the number of sites belonging to $n$-aperiodic
pairs in $E$ may be proportional to the size of $E$ itself. Actually
the situation is not quite this bad, since we need not worry about
pairs $(u,v)$ for which both $u$ and $v$ are in the same set $E$
or $F$, so, assuming we bound the distance between $u,v$ by some
$R_{n}$, we need ``only'' worry about $n$-aperiodic pairs that
lie within $R_{n}$ of the boundary of $E$. Nevertheless, since $R_{n}$
grows with $n$, when $E$ is finite the number of pairs that need
fixing easily exceeds the number of sites we can use to fix them.

Consequently, in addition to the syndeticity condition on $n$-aperiodic
pairs, we will impose a complementary ``sparseness'' condition to
ensure that the number of $n$-aperiodic pairs $(u,v)\in E\times F$
is at most proportional to the size of the boundary.
\item When one member $u\in E$ of an aperiodic pair $(u,v)\in E\times F$
is ``moved'' from $E$ to a site $u'\in(E\cup F)^{c}$, we may not
be able to ensure that $u=u'\bmod n$, so, even though we may be able
to ensure different symbols at $u'$ and $v$, the pair $(u',v)$
is not an $n$-aperiodic pair. Instead, we could try to pair $u'$
with a new site $v'$ to get an $n$-aperiodic pair, but it may happen
that all points that agree with $u'$ modulo $n$ are already part
of an aperiodic pair, or that this cannot be done without violating
the sparsity condition discussed in (1). 

To avoid this problem, we will not work with aperiodic pairs $(u,v)$
as above. Instead, introduce larger ``aperiodic sets'' $W=\{u_{1},\ldots,u_{N_{n}}\}\subseteq\mathbb{Z}^{2}$
which will be required to always contain an aperiodic pair among its
members. The point of working with a larger collection of sites is
that if $W$ is large enough then there are guaranteed to be pairs
in it lying in the same residue class mod $n$. Thus, if some element
of $W$ is replaced by another, we can hope to choose the symbol at
the new site so as to make sure that $W$ still contains an $n$-aperiodic
pair.
\end{enumerate}
We have arrived at the following strategy. For a given gap $g>0$,
we will define a subshift consisting of configurations $x$ that admit
a sequence $\mathcal{W}_{n}$, with each $\mathcal{W}_{n}$ consisting
of subsets $W\subseteq\mathbb{Z}^{2}$ of size $N_{n}$, with the
following properties:
\begin{itemize}
\item \textbf{Aperiodicity:} Every $W\in\mathcal{W}_{n}$ contains a pair
$u\neq v$ with $u=v\bmod n$ and $x_{u}\neq x_{v}$ Furthermore,
the pattern $x|_{W}$ will be ``robust'' in the sense that its properties
can be preserved despite local changes.
\item \textbf{Syndeticity: }For every closed $R_{n}$-ball $B$ there exists
$W\in\mathcal{W}_{n}$ with $W\subseteq B$.
\item \textbf{Sparsity: }For every finite region $E$, the set $\{u\in\mathbb{Z}^{2}\setminus E\mid d(u,E)\leq g\}$
is larger than $\sum_{n}|E\cap(\cup\mathcal{W}_{n})|$ .
\end{itemize}
These conditions are good approximations of the ones we will actually
use. 

It is important to keep in mind that \textbf{the structure $(\mathcal{W}_{n})$
and other auxiliary structure supporting it will not be encoded in
the configurations of the subshift}. This makes the structure flexible,
and when gluing configurations these structures may change also far
from the gluing boundary. We shall call this auxiliary structure a
\textbf{certificate}. Thus, in order to belong to our subshift, a
point must admit a compatible certificate, and our task will be to
show that, when two configurations $x,y$ are glued together to give
$z$, we can simultaneously ``glue'' their certificates to give
an admissible certificate for $z$.

\subsection{Standing notation: $a\gg b$}

Throughout the paper we write $a\gg b$ for ``$a$ sufficiently large
relative to $b$''. Thus, a statement of the form ``If $a\gg b$,
then\ldots '' means ``For every $b$, if $a$ is sufficiently large
in a manner depending on $b$, then\ldots ''. A statement such as
``Since $a\gg b$ we have\ldots '' means ``Since we are free to
choose $a$ arbitrarily large relative to $b$, we can choose it large
enough that\ldots ''.

\subsection{Organization of the paper}

In Section \ref{sec:Coins-and-buckets} we explain how to address
the robustness alluded to in paragraph (2) above. 

In Section \ref{subset:sparse-obstacle-definitions} we explain the
sparsity requirement we impose in order to satisfy paragraph (1) above.

In Section \ref{sec:Main-construction} we describe the construction
in detail.

In Section \ref{sec:Strong-irriducibilty} we show that the resulting
subshift is strongly irreducible.

Section \ref{sec:Concluding-remarks} contains further remarks and
open problems.

\subsection{Acknowledgment}

I would like to thank Joshua Frisch for bringing this problem to my
attention. I am also grateful to the anonymous referees for suggesting
many corrections to the first version of the manuscript.

\section{\label{sec:Example}Example \ref{example}}

Our goal is to show that there exists a sequence $R_{n}\rightarrow\infty$
such that $X_{(R_{n})}\subseteq\{0,1\}^{\mathbb{Z}^{2}}$ is not SI.
The subshift $X_{(R_{n})}$ was defined in the Section \ref{subsec:Discussion}.

The framework for the construction is as follows. We will define an
integer sequence $n_{k}\rightarrow\infty$ such that every integer
$\ell$ divides some $n_{k}$, and another integer sequence $r_{k}\rightarrow\infty$,
with $r_{k}$ a multiple of $n_{k}$. We say that an $r_{k}\times r_{k}$
pattern is acceptable (with respect to $r_{1},\ldots,r_{k}$ and $n_{1},\ldots,n_{k}$,
whose mention we shall usually leave implicit) if for every $m\leq k$,
every $r_{m}\times r_{m}$ subpattern of $a$ admits an $n_{m}$-aperiodic
pair. Set 
\[
R_{n}=\min\{r_{k}\mid k\in\mathbb{N}\:,\:n|n_{k}\}
\]
We then take $X_{(R_{n})}$ as be our example. 
\begin{claim}
If $x\in\{0,1\}^{2}$ contains an $r_{k}\times r_{k}$ subpattern
$b$ that is not acceptable, then $x\notin X_{(R_{n})}$. 
\end{claim}
\begin{proof}
Indeed, let $1\leq m\leq k$ be an index such that $b$ does not contain
an $n_{m}$-periodic pair. Then $R_{n_{m}}=r_{m}$, because $n_{m'}<n_{m}$
for $m'<m$, and hence $n_{m}$ cannot divide any $n_{m'}$ with $m'<m$,
and hence $b$ is a pattern on an $R_{n_{m}}\times R_{n_{m}}$-ball
that does not contain any $n_{m}$-aperiodic pairs, showing that $x\notin X_{(R_{n})}$.
\end{proof}
We now turn to the construction proper. Fix a one-dimensional aperiodic
sequence $z\in\{0,1\}^{\mathbb{N}}$ and an auxiliary sequence $s_{n}\rightarrow\infty$
with the property that every subword of length $s_{n}$ of $z$ contains
an $n$-aperiodic pair. A square pattern will be said to be $z$-lined
if the pattern along each of its sides, read in the counter-clockwise
direction and excluding the corners, is an initial segment of $z$.
Every square $N\times N$ pattern can be extended to a $z$-lined
pattern of dimensions $(N+2)\times(N+2)$. 

We now define the sequences $n_{1}<n_{2}<\ldots$ and $r_{1}<r_{2}<\ldots$
inductively, satisfying $r_{m}\geq2s_{n_{m}}+2$, together with $r_{k}\times r_{k}$-patterns
$x^{(k)}$ that are centered at the origin $0=(0,0)$, are acceptable
in the sense defined above, and such that $x^{(k)}$ extends $x^{(k-1)}$.
The definition is as follows:

Let $n_{1}=1$ and $r_{1}=2s_{1}+3$. Let $x^{(1)}$ be any centered
$r_{1}\times r_{1}$-pattern with the symbol zero at the origin (i.e.,~$x_{0}^{(1)}=0$)
and also containing the symbol $1$ (so it is $1$-aperiodic). Trivially,
$x^{(1)}$ is acceptable, since the only $r_{1}\times r_{1}$ sub-pattern
of $x^{(1)}$ is $x^{(1)}$ itself and we chose it to be $1$-aperiodic.

Suppose that $n_{1}<\ldots<n_{k}$, $r_{1}<\ldots<r_{k}$, and $x^{(1)},\ldots,x^{(k)}$
have been defined as stated. Choose the minimal $g_{k}\in\mathbb{N}$
and, given $g_{k}$, the lexicographically least pattern $b^{(k)}\in\{0,1\}^{[-g_{k},g_{k}]^{2}}$,
such that the pair $(g_{k},b^{(k)})$ satisfy 
\[
b_{0}^{(k)}=1
\]
and such that the $r_{k}\times r_{k}$ pattern $y^{(k)}$ given by
\[
y^{(k)}=\left\{ \begin{array}{cc}
y_{u}^{(k)}=x_{u}^{(k)} & u\notin[-g_{k},g_{k}]^{2}\\
b_{u}^{(k)} & u\in[-g_{k},g_{k}]^{2}
\end{array}\right.
\]
is acceptable. To justify taking the least such pair note that such
pairs $(g_{k},b^{(k)})$ exists, e.g. one is obtained by taking $g_{k}=\frac{1}{2}(r_{k}-1)$
and flipping all bits in $x^{(k)}$, so $b_{u}^{(k)}=1-x_{u}^{(k)}$. 

Let 
\[
n_{k+1}=k(r_{k}+2)
\]
(in particular, $k|n_{k+1}$), and let 
\[
r_{k+1}\geq2s_{n_{k}}+2
\]
be any odd multiple of $n_{k+1}$.

To define $x^{(k+1)}$, first extend $x^{(k)},y^{(k)}$ to $z$-lined
patterns $\widehat{x}^{(k)},\widehat{y}^{(k)}$ of dimensions $(r_{k}+2)\times(r_{k}+2)$.
Then extend $\widehat{x}^{(k)}$ to a centered $r_{k+1}\times r_{k+1}$
pattern by surrounding $\widehat{x}^{(k)}$ with translates of $\widehat{y}^{(k)}$.
\begin{claim}
$x^{(k+1)}$ is acceptable.
\end{claim}
\begin{proof}
Just observe that
\begin{itemize}
\item The only $r_{k+1}\times r_{k+1}$ sub-pattern is $x^{(k+1)}$ itself,
and $\widehat{x}_{0}^{(k)}=0\neq1=\widehat{x}_{(0,n_{k+1})}^{(k+1)}$
(because $(0,n_{k+1})$ is the center of a copy of $\widehat{y}^{(k)}$).
So $0=(0,0)$ and $(0,n_{k+1})$ are an $n_{k}$-aperiodic pair.
\begin{itemize}
\item For $1\leq m\leq k$, if $b$ is an $r_{m}\times r_{m}$ sub-pattern
of $\widehat{x}^{(k+1)}$, then either it is a sub-pattern of $x^{(k)}$
or of (a translate of) $y^{(k)}$, in which case $b$ contains an
$n_{m}$-aperiodic pair because $x^{(k)},y^{(k)}$ are acceptable,
or else $b$ intersects the boundary of $\widehat{x}^{(k)}$ or of
(a translate of) $\widehat{y}^{(k)}$ in a segment of $z$ of length
at least $r_{m}/2$, in which we are guaranteed, be definition of
$z$ and $s_{m}$, to see a initial segment of $z$ of length $\geq r_{m}/2-1=s_{m}$,
which contains an $n_{m}$-aperiodic pair.
\end{itemize}
\end{itemize}
\end{proof}
\begin{claim}
\label{claim:bk-cannot-be-inserted-in-xk-plus-1}If we replace the
central $[-g_{k},g_{k}]^{2}$ pattern in $x^{(k+1)}$ with $b^{(k)}$
(equivalently, replace the central copy of $x^{(k)}$ with $y^{(k)}$),
then the resulting pattern in no longer acceptable.
\end{claim}
\begin{proof}
The resulting pattern is tiled by copies of $\widehat{y}^{(k)}$,
whose dimensions are $n_{k+1}\times n_{k+1}$, and hence contains
no $n_{k+1}$-aperiodic pair. 
\end{proof}
\begin{cor}
The sequence $(g_{k},b^{(k)})$ is strictly increasing in the lexicographic
order on pairs $(g,b)$ with $g\in\mathbb{N}$ and $b\in\{0,1\}^{[-g,g]2}.$
In particular, $g_{k}\rightarrow\infty$.
\end{cor}
\begin{proof}
The previous claim shows that the sequence $b^{(k+1)}\neq b^{(k)}$,
and in fact, since the central $r_{m}\times r_{m}$ pattern in $x^{(k)}$
is $x^{(m)}$, it shows that $b^{(k)}\neq b^{(m)}$ for all $m<k$.
Thus $(g_{k},b^{(k)})$ does not contain any repetitions. The fact
that it is strictly increasing now follows from our choice of the
minimal candidate at each stage.
\end{proof}
Since $x^{(k+1)}$ extends $x^{(k)}$ and these are centered square
patterns of increasing side length, these patterns increases to a
configuration $x\in\{0,1\}^{\mathbb{Z}^{2}}$. 
\begin{claim}
$x\in X_{(R_{n})}$. In particular, $X_{(R_{n})}\neq\emptyset$.
\end{claim}
\begin{proof}
Given $n$, by definition there exists $k$ with $R_{n}=r_{k}$ and
$n|n_{k}$. Every $R_{n}\times R_{n}$-subpattern $a$ of $x$ is
a sub-pattern of some $x^{(m)}$ with $m>k$, and since $x^{(m)}$
is acceptable, $a$ must contain an $n_{k}$-aperiodic pair, and hence
an $n$-aperiodic pair (since $n|n_{k}$).
\end{proof}
\begin{claim}
$X_{(R_{n})}$ is not SI .
\end{claim}
\begin{proof}
Fix $g\in\mathbb{N}$ and let us show that $X_{(R_{n})}$ is not SI
with gap $g$. Specifically, we claim that if we set  $x_{0}=1$,
then there is no way to redefine $x$ on $[-g,g]^{2}\setminus\{0\}$
in a manner that the resulting configuration $x'$ belongs to $X_{(R_{n})}$. 

Indeed, let $b\in\{0,1\}^{[-g,g]^{2}}$ with $b_{0}=1$ and let $x'$
be as above. Then either $b\neq b^{(k)}$ for all $k$, which by definition
means that the central $r_{k}$-pattern in $x'$ is not acceptable
for any $k$ with $g_{k}>g$, or else $b=b^{(k)}$ for some $k$,
in which case by Claim \ref{claim:bk-cannot-be-inserted-in-xk-plus-1},
the central $r_{k+1}\times r_{k+1}$-pattern in $x'$ is not acceptable.
In both cases $x'\notin X_{(R_{n})}$.
\end{proof}
This concludes the construction.

\section{\label{sec:Coins-and-buckets}Coins and buckets (a combinatorial
lemma)}

In this section we analyze a combinatorial game which arises in our
context when $x$ is a configuration, $W\in\mathcal{W}_{n}$ are as
described at the end of the introduction, and a site from $W$ is
moved to a new location which we do not control, but we are able to
choose the symbol of $x$ at the new location. We want to ensure that
after the move the set $W$ still contains two sites that are equal
modulo $n$ but carry different symbols.

This leads us to the following model, in which the buckets represent
residue classes, coins represent sites $u\in W$, and the value $x_{u}$
is represented by the orientations of the coin (Heads or Tails; we
assume a binary symbol). We note that in our application later on,
the buckets will represent other data in addition to the residue classes.
\begin{itemize}
\item A \textbf{coin-and-bucket configuration }is an arrangement of finitely
many coins in finitely many buckets. 
\item Each coin has an \textbf{orientation }of ``heads'' or ``tails''. 
\item Configurations are modified by applying one or more \textbf{moves}
in sequence. A move\textbf{ }consists of choosing a bucket $B$ and
orientation $\sigma$, removing from $B$ a coin of orientation $\sigma$,
and placing it in a bucket $B'$ with orientation $\sigma'$, where
$\sigma'$ must be equal to the less common orientation among the
coins that were in the destination bucket before the transfer, unless
there was a tie, in which case we can specify the orientation of the
moved coin. 

A move is legal for a given configuration if the source bucket contains
a coin of the desired orientation.
\item A bucket is \textbf{oriented} if all the coins in it have the same
orientation. A configuration is oriented if all its buckets are oriented
(the orientation can depend on the bucket).
\item A configuration is \textbf{orientable} if there exists a legal sequence
of moves that, when applied to the configuration, leads to an oriented
configuration. If no such sequence exists, the configuration is \textbf{unorientable}.
\end{itemize}

Note that when coins are moved from one bucket to another we often
do not control the orientation they will receive. In particular, after
a coin is added to a non-empty bucket, that bucket is not oriented.
Observe also that that legal moves do not change the total number
of coins in a configuration, and that in each bucket the coins of
a given orientation are interchangeable, in the sense that if one
of them is to be moved, it does not matter which one it is.

One can interpret this as a game. The first player sets up the initial
configuration, and the second player applies a sequence of moves (choosing
orientations as he pleases when relevant). The second player wins
if he manages to reach an oriented configuration.

Let $c_{k,n}$ denote the configuration consisting of $k$ buckets,
all of which are empty except for the first, which contains $n_{1}=n$
coins, out of which $n_{2}=\left\lfloor n_{1}/2\right\rfloor $ are
heads and $\left\lceil n_{1}/2\right\rceil $ are tails. Consider
the following sequence of moves starting from $c_{k,n}$. First, move
$n_{2}$ heads from the first bucket to the second bucket, one by
one, orienting them as ``tails'' whenever there is a choice. Thus,
when the first coin lands in an empty bucket its orientation is tails,
the second coin is forced to be heads, the next is again tails, and
so on. When all $n_{2}$ coins have been transferred, we will have
$n_{3}=\left\lfloor n_{2}/2\right\rfloor $ heads and $\left\lceil n_{2}/2\right\rceil $
tails in the second bucket. Now move all heads from the second bucket
to the third in the same manner, and continue, ending in the $k$-th
bucket. At the end $\left\lceil n_{i}/2\right\rceil $ tails remain
in the buckets $i=1,\ldots,k-1$, and $n_{k}$ coins remain in the
last bucket. The final configuration is  oriented if and only if $n_{k}\leq1$,
and, tallying how many coins are left in each bucket, we have shown
that $c_{k,n}$ is orientable if $n\leq2^{k+1}-1$. Conversely,
\begin{prop}
\label{prop:coins-and-buckets-1}If $n\geq2^{k}$ then $c_{k,n}$
is unorientable.
\end{prop}
We begin the proof. Assume that $k,n$ are given and that $m_{1}\ldots m_{N}$
is a sequence of legal moves taking $c=c_{k,n}$ to an oriented configuration
$c'$. Our goal is to deduce that $n\leq2^{k}-1$. Our strategy will
be to reduce the given moves to a different sequence, in which each
bucket is dealt with separately; then we will be able to directly
analyze the situation similarly to the example above.

\subsubsection*{A modified game}

To simplify analysis, we first modify the rules of the game. Introduce
a special bucket called ``the pile'', which serves as a temporary
holding place for coins. While in the pile, coins do not have a specific
orientation. We only allow moves that transfer a coin from a bucket
to the pile, or from the pile to a bucket. The game begins with all
coins in the pile, and must end in an empty pile and oriented buckets.

We translate the original list of moves to this new game as follows:
\begin{enumerate}
\item Append $n$ new moves $m_{-n},\ldots,m_{-1}$ to the start of the
sequence, each one taking a coin from the pile to the first bucket,
requiring it to be set to ``tails'' if relevant.
\item An original move that transfers a coin from bucket $B$ to bucket
$B'$ is replaced by two moves: one taking a coin from $B$ to the
pile, and the other taking a coin from the pile to $B'$.
\end{enumerate}
Evidently, starting from $n$ coins on the pile and $k$ empty buckets,
the first $n$ moves in the updated sequence result in the configuration
$c$, and the remaining moves bring us to $c'$. 

We keep the same notation, using $m_{1},\ldots,m_{N}$ for the updated
sequence of moves. Our goal is still to show that $n\leq2^{k}-1$.

\subsubsection*{Final states and how to reach them}

For each bucket $B$, let 
\[
\ell_{B}=\#\text{ of coins in }B\text{ in the final configuration }c'
\]
and let $\sigma_{B}$ denote their common orientation (if $B$ is
empty in $c'$ set $\sigma_{B}=$tails). Let $\sigma'_{B}\neq\sigma_{B}$
denote the complementary orientation, and set 
\[
\ell'_{B}=\max\{\ell_{B}-1,0\}
\]
Define the sequence of moves $\underline{m}_{B}=m_{B,1},\ldots,m_{B,\ell_{B}+2\ell'_{B}}$
as follows:
\begin{itemize}
\item The first $\ell_{B}+\ell'_{B}$ moves insert coins into $B$, with
the provision that if relevant, the orientation is $\sigma_{B}$.
\item The last $\ell'_{B}$ moves remove $\sigma'_{B}$-oriented coins from
$B$.
\end{itemize}
Then, if we apply $\underline{m}_{B}$ starting from a configuration
in which $B$ is empty and there are at least $\ell_{B}+\ell'_{B}$
coins in the pile, then the net result is that $\ell_{B}$ coins are
transferred to $B$ with orientation $\sigma_{B}$. Note that no shorter
sequence of moves can achieve this result.

\subsubsection*{Peak states }

We say that $B$ is in its \textbf{peak state} if the number $p_{B}$
of $\sigma_{B}$-coins in it is maximal relative to all states it
passes through in the course of the game, and, subject to this, the
number $p'_{B}$ of $\sigma'_{B}$-coins in it is maximal. Note that
$p_{B}\geq\ell_{B}$.

Let $t_{B}$ denote the maximal integer $1\leq t_{B}\leq N$ such
that, after performing the move $m_{t_{B}}$, the bucket is in its
peak state, but was not before. Observe that the move at time $t_{B}$,
if it exists, must move a coin to $B$, so the times $t_{B}$ are
distinct for different buckets, provided that $p_{B}>0$, while if
$p_{B}=0$, then the bucket remains empty throughout, and we can delete
all such buckets and reduce $k$. Thus, we may assume all $t_{B}$
to be distinct. 

\subsubsection*{Separating moves of different buckets}

We now perform a sequence of changes on the list of moves. In order
not to confuse matters by frequent re-indexing, we think of the index
$i$ of the move $m_{i}$ as part of the move, rather than its position
in the list; so after deleting other moves and/or inserting moves,
the move will still carry the index $i$ (new moves do not receive
an index at this point). Thus, if we began with $m_{1}m_{2}m_{3}m_{3}$,
insert two new moves $m',m''$ in the middle, and delete $m_{2}$,
we are left with the sequence $m_{1}m'm''m_{3}m_{4}$, so $m_{3}$
refers to the fourth element of the sequence.

We now modify the list of moves as follows: for each bucket $B$ in
turn, replace $m_{t_{B}}$ with the sequence of moves $\underline{m}_{B}$
above, and delete all other moves involving $B$.

Then the new sequence is a list of legal moves taking the initial
configuration to $c'$. Indeed, it is clear that if the sequence is
legal, then it results in $c'$. To see that it is indeed legal, one
should note that we are moving fewer than $p_{B}+p'_{B}$ moves into
$B$ and doing so no early than $t_{B}$, so at times before $t_{B}$,
the pile is larger than it previously was; thus, there will be no
shortage of coins for moves involving other buckets, or for the moves
that put coins into $B$. Similarly, we return the maximal number
of coins to the pile after time $t_{B}$, so there will be no shortage
afterwards.

\subsubsection*{Completing the proof}

The situation now is this: starting from the configuration $c$ of
empty buckets and $n$ coins in the pile, we have $k$ sequences of
moves, $m^{(1)},\ldots,m^{(k)}$, such that 
\begin{enumerate}
\item $m^{(i)}$ moves $\ell_{B_{i}}+\ell'_{B_{i}}$ coins from the pile
to $B_{i}$ and then returning $\ell'_{B_{i}}$ coins of orientation
$\sigma_{B_{i}}$ to the pile,
\item The concatenated sequence $m=m^{(1)}m^{(2)}\ldots m^{(k)}$ takes
$c$ to the oriented configuration $c'$. 
\end{enumerate}
Let $p_{i}$, $i=1,\ldots,k$, denote the number of coins in the pile
after completing the $i$-th sequence $m^{(i)}$ and $p_{0}=n$. Then
\begin{itemize}
\item The block of moves $m^{(i)}$ starts with a pile of size $p_{i-1}$,
and it is legal and that transfers $\ell_{B_{i}}+\ell'_{B_{i}}$ coins
from the pile to $B_{i}$, so 
\[
p_{i-1}\geq\ell_{B_{i}}+\ell'_{B_{i}}
\]
\item $\ell_{B_{i}}\leq\ell'_{B_{i}}+1$ so 
\[
p_{i-1}\geq2\ell_{B_{i}}-1
\]
\item $p_{i}=p_{i-1}-\ell_{B_{i}}$, so
\begin{align*}
2p_{i} & =2p_{i-1}-2\ell_{B_{i}}\\
 & \geq2p_{i-1}-(p_{i-1}+1)\\
 & =p_{i-1}-1
\end{align*}
\item $\ell_{k}\leq1$ and $p_{k}=0$, so $p_{k-1}\leq1$. Using the relation
$p_{i}\leq2p_{i+1}+1$ from the previous step, it follows inductively
that
\[
p_{k-i}\leq2^{i-1}+2^{i}+\ldots+1=2^{i}-1
\]
and in particular
\[
n=p_{0}\leq2^{k}-1
\]
\end{itemize}
This concludes the proof.

\section{Paths avoiding sparse obstacles (more combinatorics)}

In this section we give a condition that, when satisfied by a family
of obstacles in the plane, ensures that every point sufficiently far
from the obstacles lies on an infinite, relatively flat polygonal
path that avoids the obstacles entirely, and such that this property
persists under changes to obstacles sufficiently far from the point. 

\subsection{\label{subset:sparse-obstacle-definitions}Definitions and main statement }

We introduce the following definitions. All shapes are in $\mathbb{R}^{2}$
and widths, heights etc.,~are real numbers.
\begin{itemize}
\item A \textbf{(rectilinear) rectangle }of dimensions $w\times h$ is a
translate of $[0,w]\times[0,h]$. 
\item A \textbf{(rectilinear) diamond} of dimensions $w\times h$ is a translate
of the closed convex hull of the points $(\pm w/2,0),(0,\pm h/2)$.
The points in a diamond with maximal and minimal $y$-coordinate are
called the north and south poles, and the horizontal line segment
passing through its center of mass is called its equator.
\item Given a bounded measurable set $E\subseteq\mathbb{R}^{2}$ of positive
Lebesgue measure and $c>0$, denote by $cE$ the convex set that is
a translate of $\{cx\mid x\in E\}$ and has the same center of mass
as $E$. 
\item Given a collection $\mathcal{E}$ of bounded convex sets and $c>0$,
we say that $\mathcal{E}$ is $c$\textbf{-sparse} if for every $E,E'\in\mathcal{E}$
with $E\neq E'$, we have $E\cap cE'=cE\cap E'=\emptyset$ .
\item A \textbf{polygonal graph $\gamma\subseteq\mathbb{R}^{2}$ }is the
graph of a piecewise linear function $f:\mathbb{R}\rightarrow\mathbb{R}$.
If the absolute values of slopes of all line segments in $\gamma$
is $\leq\alpha$ we say that $\gamma$ has \textbf{slope} $\leq\alpha$
(equivalently, $f$ is $\alpha$-Lipschitz).
\item A \textbf{polygonal path }will mean a polygonal graph in some (not
necessarily standard) coordinate system
\end{itemize}

The remainder of this section is devoted to the proof of the following
theorem.
\begin{prop}
\label{prop:safe-paths}Let $w_{1},\ldots,w_{N}>0$ and $h_{1},\ldots,h_{N}\geq0$.
Also denote $h_{0}=w_{0}=0$. If
\begin{align*}
h_{n} & \gg w_{n-1},h_{n-1}\hfill\text{~~~ for }n=1,\ldots,N\\
w_{n} & \gg h_{n},w_{n-1}\hfill\text{~~~ for }n=1,\ldots,N
\end{align*}
then the following holds: To every sequence $\mathcal{R}=(\mathcal{R}_{1},\ldots,\mathcal{R}_{N})$,
where each $\mathcal{R}_{n}$ is an $80$-sparse collection of rectilinear
rectangles of dimensions $w_{n}\times h_{n}$, we can associate a
set $\safe(\mathcal{R})\subseteq\mathbb{R}^{2}$, the set of \textbf{safe
points }relative to $\mathcal{R}$, such that
\begin{enumerate}
\item $\mathbb{R}^{2}\setminus\bigcup_{n=1}^{N}\bigcup_{R\in\mathcal{R}_{n}}20R\subseteq\safe(\mathcal{R}_{1},\ldots,\mathcal{R}_{N})\subseteq\mathbb{R}^{2}\setminus\bigcup_{n=1}^{N}\bigcup_{R\in\mathcal{R}_{n}}2R$
\item Every vertical line segment of length $8h_{N}$ intersects $\safe(\mathcal{R})$.
\item Every $p\in\safe(\mathcal{R})$ lies on a polygonal graph $\gamma\subseteq\safe(\mathcal{R})$
of slope $\leq h_{1}/w_{1}$.
\item For every $p\in\safe(\mathcal{R})$, if $\mathcal{R}'=(\mathcal{R}'_{1},\ldots,\mathcal{R}'_{N})$
with $\mathcal{R}'_{n}$ a $40$-sparse family of rectilinear $w_{n}\times h_{n}$
rectangles, and if $\mathcal{R}'$ agrees with $\mathcal{R}$ near
$p$ in the sense that
\[
\{R\in\mathcal{R}_{n}\mid p\in40R\}=\{R'\in\mathcal{R}'_{n}\mid p\in40R'\}
\]
then $p\in\safe(\mathcal{R}')$.
\end{enumerate}
\end{prop}
When applying the theorem, we often identify $(\mathcal{R}_{1},\ldots,\mathcal{R}_{N})$
with $\mathcal{R}=\bigcup_{n=1}^{N}\mathcal{R}_{n}$. This is an abuse
of notation since the latter does not always determine the former.
If one think of the union as a multiset (with points having multiplicity
counting how many times they appear in the union) then, with $w_{h},h_{n}$
given, there is no ambiguity.

We shall apply the proposition later to collections of rectangles
in arbitrary orthonormal coordinate systems. Since this section is
devoted to the proof of the proposition above, we work only with rectilinear
rectangles (and diamonds), and therefore omit the qualifier ``rectilinear''
for convenience.

The remainder of this section is devoted to the proof of Proposition
\ref{prop:safe-paths}.

\subsection{Jigsaws }
\begin{defn}
A \textbf{jigsaw }is a non-negative piecewise linear function whose
domain is a bounded interval in $\mathbb{R}$. An $\alpha$-jigsaw
is a jigsaw of slope $\leq\alpha$ (equivalently, $\alpha$-Lipschitz).

If $p=(x_{0},y_{0})\in\mathbb{R}^{2}$ with $y_{0}\geq0$, define
the function 
\[
\Delta_{p,\alpha}(x)=\max\{0,y_{0}-\alpha|x-x_{0}|\}
\]
This is the $\alpha$-jigsaw supported on $[x_{0}-y_{0}/\alpha,x_{0}+y_{0}/\alpha]$.
Its graph on this interval is an isosceles triangle whose base is
the interval above, whose sides are of slope $\pm\alpha$, and whose
peak is at $p$.
\end{defn}
\begin{lem}
\label{lem:bound-on-adjoining-effect}Let $f$ be a jigsaw and $p=(x_{0},y_{0})\in\mathbb{R}^{2}$
with $y_{0}\geq f(x_{0})$. Suppose that $\beta>\alpha$ and set 
\[
r=\frac{y_{0}-f(x_{0})}{\beta-\alpha}
\]
Suppose that $f$ has slope at most $\alpha$ on $I=[x_{0}-r,x_{0}+r]$.
Then 
\[
g=\max\{f,\Delta_{p,\beta}\}
\]
is a jigsaw such that
\begin{itemize}
\item $f\leq g\leq f+\Delta_{(x_{0},y_{0}-f(x_{0})),\beta-\alpha}$ .
\item $g|_{\mathbb{R}\setminus I}=f|_{\mathbb{R}\setminus I}$ .
\item The slope of $g$ on $I$ is bounded by $\beta$ .
\end{itemize}
\end{lem}
We omit the elementary proof.

\subsection{Double jigsaws}

We identify a jigsaw with the closed set that it bounds with the $x$-axis.
More precisely,
\begin{defn}
Given a jigsaw $g$ whose minimal closed supporting interval is $I$,
let 
\[
A(g)=\{(x,y)\mid x\in I\;,\;0\leq y\leq g(x)\}
\]
If $h$ is a function such that $-h$ is a jigsaw, define $A(h)$
similarly:
\[
A(h)=\{(x,-y)\mid(x,y)\in A(-h)\}
\]
\end{defn}
By joining regions of the types $A(g),A(h)$ above one can form more
complicated shapes (e.g. diamonds):
\begin{defn}
A \textbf{double jigsaw $G\subseteq\mathbb{R}^{2}$ with equator $J\subseteq\mathbb{R}^{2}$
}is a set of the form $G=(A(g^{+})\cup A(-g^{-}))+p$ and an interval
$J=I+p$, where
\begin{itemize}
\item $p\in\mathbb{R}^{2}$ .
\item $g^{+},g^{-}$ are jigsaws.
\item $I$ is the minimal closed  interval supporting $g^{+},g^{-}$.
\end{itemize}
If $g^{+},g^{-}$ are $\alpha$-jigsaws then $G$ is called a double
$\alpha$-jigsaw.
\end{defn}
Note that $G$ and $J$ and $p$ determine $I,g^{+},g^{-}$ uniquely,
and vice versa.

\subsection{The hull of collections of diamonds}

Let $\lor$ denote the maximum operator between numbers and functions,
and observe that $f\lor g$ is jigsaw whenever $f,g$ are.
\begin{defn}
[Adjoining a diamond to a double jigsaw] Let $I=[-r,r]$ be a symmetric
interval and $g^{+},g^{-}$ jigsaws supported on $I$, and let $G=A(g^{+})\cup A(-g^{-})$
denote the corresponding double jigsaw. 

Let $D$ be an $\alpha$-diamond (we do not require the equator to
be on the $x$-axis). Let $q^{+},q^{-}$ be the north and south poles
of $D$ respectively. Then the double jigsaw obtained by \textbf{adjoining
}$D$ to $G$, denoted $G\ltimes D$, is the set $A(h^{+})\cup A(-h^{-})$,
where
\begin{itemize}
\item $h^{+}=g^{+}\lor\Delta_{q^{+},\alpha}$ if $q^{+}$ is in the upper
half plane, and $h^{+}=g^{+}$ otherwise.
\item $h^{-}=g^{-}\lor\Delta_{-q^{-},\alpha}$ if $q^{-}$ is in the lower
half plane, and $h^{-}=g^{-}$ otherwise.
\end{itemize}
When the equator $I\subseteq\mathbb{R}^{2}$ of $G$ is not contained
in the $x$-axis, let $p$ be the center of $I$, and set 
\[
G\ltimes D=((G-p)\ltimes(D-p))+p
\]
\end{defn}

\begin{rem}
Note that
\begin{enumerate}
\item $G,D\subseteq G\ltimes D$.
\item The equator of $G\ltimes D$ contains the equator of $G$.
\item The operation $\ltimes$ is asymmetric, because the right-hand term
must be a diamond, while the left can be a jigsaw. Even when both
arguments are diamonds, it is not commutative.
\end{enumerate}
\end{rem}
\begin{defn}
[The hull of a family of diamonds]\label{def:hulls}Let $\mathcal{D}_{1},\ldots,\mathcal{D}_{N}$
be families of diamonds, with all elements of $\mathcal{D}_{n}$ having
dimensions $w_{n}\times h_{n}$ with $w_{1}<w_{2}<\ldots<w_{N}$.

The \textbf{hull} $\mathcal{H}=\mathcal{H}(\mathcal{D}_{1},\ldots,\mathcal{D}_{N})$
of the sequence is the collection $\mathcal{H}$ of double jigsaws
constructed as follows:
\begin{quote}
Initialize $\mathcal{H}=\emptyset$.

For each $n$ from $N$ down to $1$, 
\begin{quote}
For each $D\in\mathcal{D}_{n}$ , 
\begin{enumerate}
\item [(a)] If $d(D,G)>h_{n}$ for all $G\in\mathcal{H}$, add $D$ to
$\mathcal{H}$.
\item [(b)] Otherwise, choose a set $G\in\mathcal{H}$ with $d(D,G)\leq h_{n}$,
delete $G$ from $\mathcal{H}$ and replace it with $G\ltimes D$.
\end{enumerate}
\end{quote}
\end{quote}
\end{defn}
\begin{figure}
\includegraphics[scale=0.6]{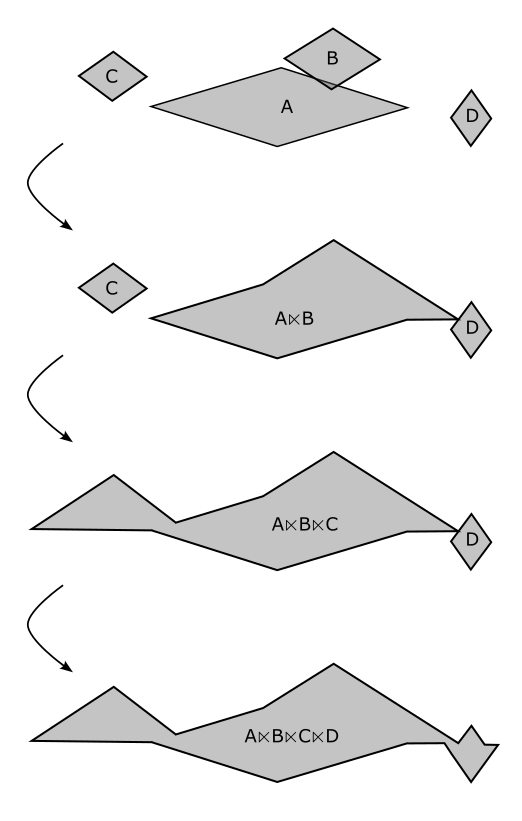}\caption{Merging diamonds to create their hull.}

\end{figure}

There are many possible orders with which to choose the diamonds in
round $n$ of the construction, and there may be more than one choice
for $G$ in step (b). We assume for now that these choices have all
been made according to some rule that is fixed in advance, so as to
make the hull well defined. The next proposition shows that, under
some separation assumptions, these choices do not affect the outcome.
\begin{prop}
\label{prop:properties-of-hulls}Let $\mathcal{D}_{1},\ldots,\mathcal{D}_{N},$$w_{n},h_{n}$
be as in Definition \ref{def:hulls}. Let $\alpha_{n}=h_{n}/w_{n}$,
so $\mathcal{D}_{n}$ consists of $\alpha_{n}$-diamonds, write $h_{\leq n}=\sum_{i\leq n}h_{i}$
and $w_{\leq n}=\sum_{i\leq n}w_{i}$. Suppose that 
\begin{enumerate}
\item [(a)]$\mathcal{D}_{n}$ is $20$-sparse. 
\item [(b)] $h_{n}>10h_{\leq n-1}$ and $w_{n}>10w_{\leq n-1}$ for all
$1<n\leq N$.
\item [(c)] $\alpha_{n}<\frac{1}{10}\alpha_{n-1}$ for $n>1$. 
\end{enumerate}
Then 
\begin{enumerate}
\item If $D\in\mathcal{D}_{n}$ is merged with a double jigsaw $J\in\mathcal{H}$
is step (b) of Definition \ref{def:hulls}, then $(J\ltimes D)\setminus J$
is properly contained in a rectangle of dimensions $5w_{n}\times5h_{n}$
that contains $D$.

In particular, since $\mathcal{D}_{n}$ is $20$-separated, the order
with which diamonds are chosen in round $n$ of the construction does
not affect whether part (a) or (b) are invoked for a given diamond.
\item If $D\in\mathcal{D}_{n}$ was added in step (a) of round $n$ of Definition
\ref{def:hulls}, then in the final hull, the double jigsaw $J$ containing
$D$ is properly contained in a rectangle of dimensions $(w_{n}+10w_{\leq n-1})\times(h_{n}+10h_{n-1})$,
and hence in a rectangle of dimensions $2w_{n}\times2h_{n}$ . 
\item The double jigsaws in $\mathcal{H}$ are pairwise disjoint, and every
vertical line segment of length $2h_{N}$ contains points in the complement
of $\mathcal{H}$. 
\item If $p\in\mathbb{R}^{2}\setminus\cup\mathcal{H}$, then there exists
a polygonal graph $\gamma$ passing through $p$, disjoint from $\mathcal{H}$,
and with all line segments in the path having slope at most $\alpha_{1}$
in absolute value. 
\item Let $p\in\mathbb{R}^{2}\setminus\cup\mathcal{H}$. Suppose that $\mathcal{D}'_{1},\ldots,\mathcal{D}'_{N}$
are collections of diamonds satisfying the same conditions as $\mathcal{D}_{1},\ldots,\mathcal{D}_{N}$
and agreeing near $p$ with the original collection, in the sense
that for each $1\leq n\leq N$, 
\[
\{D\in\mathcal{D}_{n}\mid p\in10D\}=\{D'\in\mathcal{D}'_{n}\mid p\in10D'\}
\]
Then $p\in\mathbb{R}^{2}\setminus\mathcal{H}'$, where $\mathcal{H}'$
is the hull of $\mathcal{D}'_{1},\ldots,\mathcal{D}'_{N}$ .
\end{enumerate}
\end{prop}
\begin{proof}
(1) First, By induction, when the construction enters stage $n$,
every $J\in\mathcal{H}$ has slopes at most $\alpha_{n+1}$ (or $0$
when $n=N$). 

For each $n$, we induct on the sequence $(D_{i})$ of elements of
$\mathcal{D}_{n}$ processed in stage $n$ of the construction. Consider
what happens when we reach $D_{i}$, and its distance from some double
jigsaw $J\in\mathcal{H}$ is less than $h_{n}$. Let $(x_{0},y_{0})$
denote the center of $D_{i}$. Observe that $J$ arose by adjoining
zero or more of the diamonds $D_{1},\ldots,D_{i-1}$ to a double jigsaw
$J'$ that was in $\mathcal{H}$ at the beginning of stage $n$, and,
by our induction hypothesis, each time one of these $D_{j}$ was adjoined,
the change to $J'$ was limited to a box of dimensions $5w_{n}\times5h_{n}$;
since $\mathcal{D}_{n}$ is $20$-sparse, these changes are disjoint
from the box of the same dimensions with center $(x_{0},y_{0})$.
Thus, the jigsaws defining $J$ agree with those defining $J'$ on
the interval $[x_{0}-5w_{n},x_{0}+5w_{n}]$, and in particular its
slopes on this interval are at most $\alpha_{n+1}$. This implies,
in particular, that the closest point in $D_{i}$ to $J$ is its top
or bottom vertex, or that these points are on different sides of the
(line extending the) equator of $J$. Therefore, the distance of the
other vertex from the jigsaw, or the equatorial line, is at most $2h_{n}$.
Now we invoke Lemma \ref{lem:bound-on-adjoining-effect} to see that
adjoining $D_{i}$ affects $J$ only in the interval $[x_{0}-r,x_{0}+r]$
for $r=2h_{n}/(\alpha_{n}-\alpha_{n+1})$, and since all changes happen
in a region whose boundary has slopes $\alpha_{n}$, we have confined
the change to a box of dimension $2r\times(\alpha_{n}\cdot2r)$. Using
the definition of $r$, and the assumptions $h_{n}=\alpha_{n}w_{n}$
and $\alpha_{n+1}<\alpha_{n}/10$, we see that the changes are confined
to a box of dimensions $5w_{n}\times5h_{n}$ containing $D_{i}$.

The second part of (1) follows from the first part and the sparsity
assumption, because together they imply that the region added to the
jigsaw when $D_{i}\in\mathcal{D}_{n}$ is added does not come within
$h_{n}$ of any other diamond in $\mathcal{D}_{n}$, and so does not
affect the choice of option (a) or (b) in future stages of the construction.

(2) follows because (1) shows that the additions in stage $n$ do
not interact, and so if $J$ is a double jigsaw in $\mathcal{H}$
before entering stage $n$, then it is extended by at most $5h_{n}$
upwards and downwards, and $5w_{n}$ left and right, in the course
of stage $n$. Each double jigsaw in the hull started from a diamond
$D\in\mathcal{D}_{n}$ for some $n$, whose dimensions are $w_{n}\times h_{n}$.
Thus the cumulative vertical height of a double jigsaw is $<h_{n}+10h_{\leq n-1}<2h_{n}$
and the cumulative width is $<w_{n}+10w_{\leq n-1}<2w_{n}$. 

(3) The first part follows from (2), using $20$-sparsity. The second
part follows from (2) and the first part of (3), since they imply
that any vertical line intersects elements of $\mathcal{H}$ in disjoint
intervals of length $\leq2h_{N}$.

(4) If $p$ is in the complement of $\cup\mathcal{H}$, we can form
a polygonal path through $p$ by traveling horizontally in either
direction; if we hit a double jigsaw in $\mathcal{H}$, we follow
its boundary until it is possible to continue horizontally in the
original direction, at which point we do so. This path is not quite
disjoint from $\cup\mathcal{H}$, since part of it may lie in the
boundary of $\cup\mathcal{H}$, but we can perturb the path by slightly
raising or lowering each line segment that follows a boundary of an
element of $\mathcal{H}$ in order to avoid the boundary entirely. 

(5) This again follows from (1), because in the course of the construction
of the hull, any diamond $D\in\mathcal{D}'_{n}\setminus\mathcal{D}_{n}$
is far enough away that adjoining cannot capture $p$ nor cause $p$
to be closer than $h_{n}$ to an element of $\mathcal{H}$ if it was
not this close already.
\end{proof}

\subsection{\label{subsec:Rectangles-and-safe-points}Application to rectangles}
\begin{proof}
[Proof of proposition \ref{prop:safe-paths}] If $R$ is a rectangle
of dimensions $w\times h$, let $\diamondsuit(R)$ denote the closed
diamond of dimensions $2w\times2h$ with the same center as $R$,
whose diagonals are parallel to the sides of $R$, and that contains
$R$.

Let $\mathcal{R}_{1},\ldots,\mathcal{R}_{N}$ be as in the statement
and define 
\[
\mathcal{D}_{n}=\{\diamondsuit(2R)\mid R\in\mathcal{R}_{n}\}
\]
Then $D\in\mathcal{D}_{n}$ are $20$-sparse sets of diamonds of dimensions
$4w_{n}\times4h_{n}$. Define $\safe(\mathcal{R}_{1},\ldots,\mathcal{R}_{N})=\mathcal{H}(\mathcal{D}_{1},\ldots,\mathcal{D}_{N})$.
Statements (2)--(4) are now immediate consequences of the corresponding
statements in Proposition \ref{prop:properties-of-hulls}. 

For (1), note that since $\safe(\mathcal{R}_{1},\ldots,\mathcal{R}_{N})$
is disjoint from all $D\in\mathcal{D}_{n}$ it is disjoint from $2R$
for all $R\in\mathcal{R}_{n}$, leading to the right inclusion in
(1). For the left inclusion, note that by part (1) of Proposition
\ref{prop:properties-of-hulls}, for each $D\in\mathcal{D}_{n}$ there
is a $5w_{n}\times5h_{n}$ rectangle $\widetilde{R}(D)$ containing
$D$ such that
\[
\mathcal{H}(\mathcal{D}_{1},\ldots,\mathcal{D}_{N})\subseteq\bigcup_{n=1}^{N}\bigcup_{D\in\mathcal{D}_{N}}\widetilde{R}(D)
\]
Since $\widetilde{R}(D)\subseteq10D$, and since there is a rectangle
$R\in\mathcal{R}_{n}$ such that $D=2R$, and hence $\widetilde{R}(D)\subseteq20R$,
we conclude that
\[
\mathcal{H}(\mathcal{D}_{1},\ldots,\mathcal{D}_{N})\subseteq\bigcup_{n=1}^{N}\bigcup_{R\in\mathcal{R}_{n}}20R
\]
This implies the left inclusion in (1).
\end{proof}

\section{\label{sec:Main-construction}Main construction}

\subsection{Notation and definitions}
\begin{itemize}
\item Elements of $\mathbb{R}^{2}$ are called \textbf{points}.

Elements of $\mathbb{Z}^{2}$ are called \textbf{sites}. 
\item A \textbf{disk }$D\subseteq\mathbb{R}^{2}$ is a closed Euclidean
ball. If the radius of a disk is $r$, we call it an $r$-disk.
\item A set $E\subseteq\mathbb{R}^{2}$ is \textbf{$r$-separated }if $\left\Vert x-y\right\Vert \geq r$
for all distinct $x,y\in E$.
\item A set $E\subseteq\mathbb{R}^{2}$ is \textbf{$R$-dense }if it intersects
every $R$-disk.
\end{itemize}
The following definition is not standard:
\begin{itemize}
\item Let $u\in\mathbb{R}^{2}$ be a unit vector and $D$ an $r$-disk.
An \textbf{almost radial rectangle} $R$ with \textbf{orientation}
$u$ (relative to the disk $D$) is a rectangle satisfying
\begin{enumerate}
\item $R\subseteq D$.
\item $R\cap\partial D\neq\emptyset$.
\item The long sides of $R$ are have length $\leq r$, are parallel to
$u$, and when they are extended in direction $u$, each of the resulting
rays passes within $\frac{1}{100}r$ of the center of $D$.
\end{enumerate}
The orientation of $R$ is 
\[
\theta(R)=u
\]
Note that $\theta(u)$ points approximately from $R\cap\partial D$
to the center of $D$. 
\end{itemize}

\subsection{Parameters}

For each $n\in\mathbb{N}$, define the parameter 
\[
N_{n}=2^{1,000,000\cdot n^{2}}
\]
By Proposition \ref{prop:coins-and-buckets-1}, there exist unorientable
coin-and-bucket configurations with $1000^{2}n^{2}$ buckets. We introduce
two sequences of real parameters 
\begin{align*}
h_{1} & \leq h_{2}\leq\ldots\\
w_{1} & \leq w_{2}\leq\ldots
\end{align*}
representing the height and width of rectangles that will appear in
the construction, and
\[
r_{n}=\frac{10}{9}w_{n}
\]
representing radii of disks. We assume further that 
\[
\frac{h_{1}}{w_{1}}\ll1
\]
and that
\begin{align*}
h_{n+1} & \gg w_{n},h_{n},n\\
w_{n+1} & \gg h_{n+1},w_{n},n
\end{align*}

\subsection{Certificates and their parts}

Recall from the introduction that our goal is to define a structure
called a certificate, consisting of 
\begin{itemize}
\item A sequence $\mathcal{W}_{1},\mathcal{W}_{2},\ldots$, where
\item Each $\mathcal{W}_{n}$ is a (``syndetic'' and ``separated'')
collection of sets $W\subseteq\mathbb{R}^{2}$, and
\item Each $W$ is a finite set whose elements are called \textbf{witnesses}.

The coordinates of $w\in W$ are real, but from $w$ we obtain a site
in $\mathbb{Z}^{2}$ by taking the coordinate-wise integer part. 
\end{itemize}
To help coordinate this data we shall introduce additional geometric
structures:
\begin{itemize}
\item Each $W\in\mathcal{W}_{n}$ will be contained in a Euclidean disk
called a \textbf{frame}.
\item Each witness $w\in W$ will be contained in an associates \textbf{box},
which is an almost radial rectangle relative to the frame that $w$
belongs to (it is long and thin and approximates a radius of the frame). 
\item Each box will be further subdivided into rectangular \textbf{sections}. 
\end{itemize}
Here, finally, is the full definition. A \textbf{certificate }consists
of
\begin{itemize}
\item Countably many \textbf{levels}, numbered $n=1,2,3,\ldots$ (corresponding
to the sets $\mathcal{W}_{n})$.
\item The $n$-th level consists of countably many $r_{n}$-discs, called
$n$\textbf{-frames}.
\begin{itemize}
\item The \textbf{center }of an $n$-frame $F$ is denoted $c(F)$ .
\item The set of centers of $n$-frames in a certificate is $\frac{1}{2}r_{n}$-separated.
\end{itemize}
If the set of centers of $n$-frames is $10r_{n}$-dense, we say that
the $n$-th level is \textbf{dense}. The certificate is dense if all
its levels are dense. We allow non-dense certificates, which arise
naturally in the intermediate stages of constructing a certificate.
\item Each $n$-frame $F$ contains $N_{n}$ rectangles $B_{1},\ldots,B_{N_{n}}$called
$n$-\textbf{boxes}, satisfying
\begin{itemize}
\item The dimensions of $n$-boxes is $w_{n}\times h_{n}$. 
\item The $B_{i}$ are almost-radial relative to $F$.
\item Every two distinct boxes $B_{i}\neq B_{j}$ satisfy $d(B_{i},B_{j})\geq100h_{n}$.
\item All the $B_{i}$ in a frame have a common orientation, and furthermore
the orientations lie in the set 
\[
\Theta=\{\frac{2\pi k}{1000}\mid0\leq k<1000\}
\]
 
\item Each $n$-box $B$ associated to a frame $F$ is divided into $1000$
closed rectangular \textbf{sections }$B^{(1)},B^{(2)},\ldots,B^{(1000)}$
of dimensions $(w_{n}/1000)\times h_{n}$, numbered by increasing
distance from the center of $F$.

We require the sections to be closed because later a limiting procedure
will require it. Evidently as described above the sections cannot
be disjoint, but one should think of them as such. Formally one can
insert minuscule spaces between them, while slightly reducing their
length in the radial direction. We do not bother to keep track of
this detail which presents no serious problems.

.
\item The set of indices of sections is denoted
\[
\Sigma=\{1,\ldots,1000\}
\]
and the section of a point $u\in B$ is denoted 
\[
\sigma(u)=\sigma_{B}(u)\in\Sigma
\]
so for $u\in B$ we always have $u\in B^{(\sigma(u))}$. Note that
$u\in\mathbb{R}^{2}$ may belong to several boxes (from different
frames and levels) so $\sigma(u)$ is not well defined, hence the
notation $\sigma_{B}(u)$. However, when the box is clear from the
context, we omit the subscript and write $\sigma(u)$.
\end{itemize}
\item Each box $B$ is associated with one\textbf{ witness $w=w(B)\in B$}
(Note: this is a point in $\mathbb{R}^{2}$, not $\mathbb{Z}^{2}$).

\begin{itemize}
\item The witness $w(B)$ must be safe (in the sense of Proposition \ref{prop:safe-paths})
with respect to the collection of rectangles 
\[
\mathcal{R}_{C}(B)=\{B',B''\}\cup\left\{ R^{(i)}\left|\begin{array}{c}
\text{\ensuremath{R}\text{ is a \ensuremath{k}-box for }}k<n\text{ such that }\theta(R)=\theta(B)\\
\text{and }2R^{(i)}\cap B^{(i)}\neq\emptyset\text{ for some }i\in\{1,\ldots,1000\}
\end{array}\right.\right\} 
\]
Here, $B',B''$ denote the rectangles of dimensions $w_{n}\times\frac{1}{100}h_{n}$
that share one long side with $B$ and whose interiors are disjoint
from $B$. We note that this collection is $80$-sparse, so it satisfies
the hypothesis of Proposition \ref{prop:safe-paths}.
\item The orientation of the box $B$ associated to $w$ is denoted $\theta(w)=\theta(B)$.
\item The section $\sigma(w)=\sigma_{B}(w)$ is also defined, as above.
\end{itemize}
\end{itemize}
\begin{figure}
\includegraphics[scale=0.6]{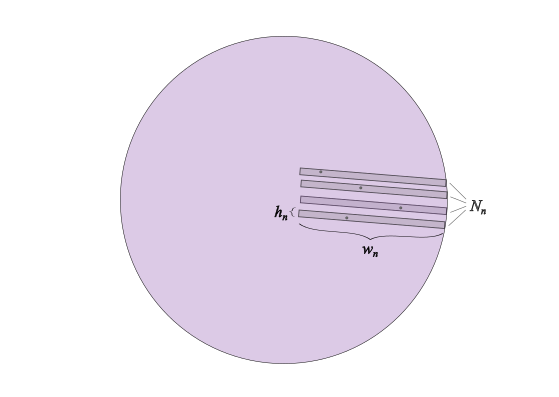}\caption{A frame and its boxes. Each box has a witness in it indicated by a
dot. We have not depicted the sections.}
\end{figure}

\begin{figure}
\includegraphics[scale=0.8]{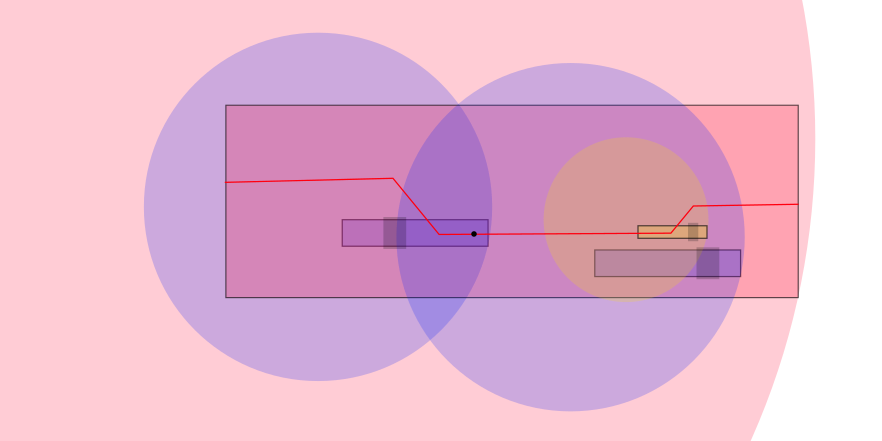}

\caption{Part of a frame $F$ and a box $B$ in $F$ (indicated in red), and
two lower level frames with boxes in the same orientation (indicated
in blue and yellow). The gray rectangles indicate sections of the
lower-level boxes which would be included in $\mathcal{R}_{C}(B)$.
The witness $w$ of $B$ is indicated by a black dot, and the red
path passing through $w$ consists entirely of safe points relative
to $\mathcal{R}_{C}(B)$.}
\end{figure}

\begin{itemize}
\item The set of all witnesses associated to a frame $F$ is denoted 
\[
W(F)=\{w\mid w\text{ is a witness in }F\}
\]
and the sets $\mathcal{W}_{1},\mathcal{W}_{2},\ldots$ associated
to the certificate $C$ are 
\[
\mathcal{W}_{n}(C)=\{W(F)\mid F\text{ is an \ensuremath{n}-frame in }C\}
\]

\end{itemize}

\subsection{Comments and first corollaries of the definition}

\subsubsection*{Notation}

We will be flexible in our notation: A witness may well belong, as
a point in $\mathbb{R}^{2}$, to many boxes, including more than one
on each level, but we assume that witnesses ``remember'' which box
they came from. Similarly, by convention, each box ``remembers''
which frame it is associated to. We shall write $w\in B$ and $B\in F$
for these associations, and $w\in F$ to indicate that $w$ is a witness
associated to some box in $F$.

\subsubsection*{Separation of witnesses}

Ideally, we would like different witnesses to occupy different sites,
but it is convenient to aim for something slightly weaker: that each
site can host at most a bounded number of witnesses.

We implement this as follows. Each witness belongs to some section
$S$, which has an associated position in its box ($\sigma(S)\in\Sigma$),
and the box has an orientation ($\theta(S)\in\Theta$). This partitions
witnesses into $|\Theta|\cdot|\Sigma|=1000^{2}$ \textbf{classes}
parameterized by $\Theta\times\Sigma$. One should think of these
classes as occupying different ``layers'' of the plane, avoiding
each other. We must still ensure that witnesses from each class do
not come close to each other. This is done in the following lemmas:
\begin{lem}
\label{lem:separation-of-boxes-from-different-frames-in-same-level}If
$F,F'$ are $n$-frames whose centers are at least $r_{n}/2$ separated,
and if $B\in F$, $B'\in F'$ are boxes with $\theta(B)=\theta(B')$,
then $d(B^{(i)},(B')^{(i)})>\frac{1}{4}r_{n}$ for all $i\in\{1,\ldots,1000\}$
.
\end{lem}
\begin{proof}
Let $p,p'$ denote the points in $B,B'$ closest to the centers of
$F,F'$ respectively. Since the rectangles are almost radial, $d(p,c(F))<r_{n}/100$
and $d(p',c(F'))<r_{n}/100$. Since $d(c(F),c(F'))>r_{n}/2$, we have
$d(p,p')>r_{n}/3$. Each section is of dimensions $(w_{n}/1000)\times h_{n}$
and $h_{n}\ll w_{n}=\frac{9}{10}r_{n}$, so (assuming as we may that
$h_{n}<w_{n}/1000$), the diameter $d$ of a section is at most $d\leq\sqrt{2}w_{n}/1000<r_{n}/100$.
Finally, $\theta(B)=\theta(B')$, so for each $i\in\Sigma$ there
are points $p_{i}\in B^{(i)}$ and $p'_{i}\in(B')^{(i)}$ such that
$d(p_{i},p'_{i})=d(p,p')>r_{n}/3$. It follows that 
\[
d(B^{(i)},(B')^{(i)})>d(p,p')-2d>r_{n}/4
\]
as claimed.
\end{proof}
\begin{lem}
\label{lem:separation-of-witnesses-with-identical-orientations-and-section}Let
$C$ be a certificate and $w,w'\in C$ distinct witnesses belonging
to levels $n\neq n'$, respectively. If $\sigma(w)=\sigma(w')$ and
$\theta(w)=\theta(w')$, then $d(w,w')\geq h_{\min\{n,n'\}}$.
\end{lem}
\begin{proof}
Without loss of generality assume $n'>n$. Write $i=\sigma(w)=\sigma(w')$.
Let $B,B'$ denote the boxes and $S=B^{(i)},S'=(B')^{(i)}$ the sections
containing $w,w'$ respectively. Since $w'$ is safe relative to $\mathcal{R}_{C}(B')$,
there are two possibilities:
\begin{itemize}
\item $S\notin\mathcal{R}_{C}(B')$, meaning that $2S\cap S'=\emptyset$,
and in particular $w'\notin2S$.
\item $S\in\mathcal{R}_{C}(B')$ Then, since $w'$ is safe with respect
to $\mathcal{R}_{C}(B')$, by part (1) of Proposition \ref{prop:safe-paths},
we know that $w'\notin2S$.
\end{itemize}
In both cases we have $w'\notin2S$, hence (since $S$ is a box of
dimensions $\frac{w_{n}}{1000}\times h_{n}$ and using $w_{n}\gg h_{n})$
we have $d(w',w)\geq h_{n}$.
\end{proof}
\begin{cor}
\label{cor:witnesses-differ-in-position-direction-or-section}For
every pair $w\neq w'$ of distinct witnesses in a certificate, at
least one of the following hold:
\begin{itemize}
\item $\theta(w)\neq\theta(w')$ .
\item $\sigma(w)\neq\sigma(w')$ .
\item $d(w,w')\geq h_{1}$.
\end{itemize}
\end{cor}
\begin{proof}
If $w,w'$ come from the same frame, and $\theta(w)=\theta(w')$,
then their boxes are at least $100h_{n}$ apart, so their distance
is certainly at least $h_{1}$.

If $w,w'$ come from different frames in the same level, then either
$\theta(w)\neq\theta(w')$, or else, by Lemma \ref{lem:separation-of-boxes-from-different-frames-in-same-level},
 $\sigma(w)\neq\sigma(w')$.

Finally, if $w,w'$ come from different levels $n\neq n'$, and Lemma
\ref{lem:separation-of-witnesses-with-identical-orientations-and-section}
says that if their directions and sections agree then $d(w,w')\geq4h_{\min\{w,w'\}}>h_{1}$.
\end{proof}

\subsubsection*{Inserting new frames into a certificate}

If $C$ is a certificate, and if $u$ is a point that is $\frac{1}{2}r_{n}$-far
from the centers of all $n$-frames in $C$, then it may be possible
to add an $n$-frame $F$ with center $u$ to $C$, but for this we
must also define new boxes and witnesses in $F$. Observe that 
\begin{itemize}
\item In order to define a new box $B\in F$, it (or, rather, certain of
its sections) must not come too close to witnesses from higher levels.
\item In order to define a new witness $w\in B$, we need to stay away from
(certain sections of) boxes from lower levels.
\end{itemize}
Thus, the task of defining a frame is split into two essentially independent
parts.

The following lemma will be useful when we need to define boxes in
a new frame, and want them to intersect a set $E$ (the set where
we want to place witnesses, because we have freedom to define the
symbols there).
\begin{lem}
\label{lem:existence-of-many-almost-radial-rectangles}For every $n$-frame
$F$ and every path-connected set $E\subseteq F\setminus$$\frac{1}{5}F$
 of diameter $>\frac{1}{100}r_{n}$, there are $N_{n}+1$ rectangles
$R_{1},\ldots,R_{N_{n}+1}$ such that
\begin{enumerate}
\item The $R_{i}$ are almost radial relative to $F$.
\item $d(R_{i},R_{j})>100h_{n}$ for $i\neq j$.
\item All $\theta(R_{i})$ are equal and are in $\Theta$.
\item For each $i$ the set $E\cap R_{i}$ contains a connected component
that intersects both long edges of $R_{i}$.
\end{enumerate}
\end{lem}
\begin{proof}
Without loss of generality we can assume that $E$ is a simple path.
Choose $u,v\in\Theta$ that differ by an angle of $2\pi/1000$ and
such that the projection of $E$ to $u^{\perp}$ and $v^{\perp}$
both come within $r_{n}/50$ of the corresponding projection of $c(F)$.
It is elementary to see that there is a universal constant $c>0$
such that one of these projections of $E$ is an interval of length
at least $cr_{n}$. Suppose $u$ is this direction. Since $r_{n}\gg h_{n},n$
we may assume the projection of $E$ in direction $u^{\perp}$ has
length is at least $1000(N_{n}+1)h_{n}$. This allows us to construct
$N_{n}+1$ almost radial rectangles in direction $u$ satisfying (1)-(3)
and such that their projection to $u^{\perp}$ lies inside the projection
of $E$, and this ensures also (4).
\end{proof}

\subsection{The alphabet}

Recall that $\Theta$ is the set of angles that are multiples of $2\pi/1000$
and $\Sigma=\{1,\ldots,1000\}$ is the index set of the sections of
boxes. Our alphabet will be 
\[
A=\{H,T\}^{\Theta\times\Sigma}
\]
where $H,T$ represent ``heads'', ``tails'' respectively. For
$(\theta,\sigma)\in\Theta\times\Sigma$ we write 
\[
\pi_{(\theta,\sigma)}:A\rightarrow\{H,T\}
\]
 for the corresponding coordinate projection, and extend $\pi_{(\theta,i)}$
to a projection $A^{\mathbb{Z}^{2}}\rightarrow\{H,T\}^{\mathbb{Z}^{2}}$. 

One should think of a configuration $x\in A^{\mathbb{Z}^{2}}$ as
consisting of $1000^{2}$ ``layers'' $\pi_{\theta,\sigma}x\in\{H,T\}^{\mathbb{Z}^{2}}$,
each layer indexed by $\Theta\times\Sigma$ and corresponding classes
of sections determined by orientation ($\theta\in\Theta$) and section
number ($\sigma\in\Sigma$). 

\subsection{Compatibility}

Recall that we have defined witnesses to be points in $\mathbb{R}^{2}$,
but would like to use them to index configurations in $\mathbb{Z}^{2}$.
To this end, for $w=(w_{1},w_{2})\in\mathbb{R}^{2}$ write $[w]=([w_{1}],[w_{2}])$,
where $[\cdot]$ denotes the integer part. Then we define
\[
x_{w}=x_{[w]}
\]

If $x\in A^{\mathbb{Z}^{2}}$ is a configuration and $F$ is an $n$-frame,
we define a \textbf{coin and bucket configuration $CBC(x,F)$} as
follows:
\begin{itemize}
\item The number of buckets is $n^{2}\cdot|\Theta|\cdot|\Sigma|$, indexed
by $(\mathbb{Z}/n\mathbb{Z})^{2}\times\Theta\times\Sigma$.
\item For every witness $w\in F$ there is associated a coin with orientation
$\pi_{(\theta(w),\sigma(w))}(x_{w})\in\{H,T\}$ in the bucket indexed
by $([w]\bmod n,\theta(w),\sigma(w))$.

Since each $n$-frame contains $N_{n}$ witnesses, there are altogether
$N_{n}$ coins in $CBC(x,F)$.
\end{itemize}
We say that $x\in A^{\mathbb{Z}^{2}}$ is: 
\begin{itemize}
\item \textbf{Compatible with a frame }$F$ if $CBC(x,F)$ is unorientable
in the sense of Section \ref{sec:Coins-and-buckets}.
\item \textbf{Compatible with a certificate }$C$ if it is compatible with
all frames (from all levels) of $C$.
\end{itemize}

\subsection{Definition of the subshift}

Let
\[
X=\left\{ x\in A^{\mathbb{Z}^{2}}\mid\text{There exists a dense certificate compatible with }x\right\} 
\]

\begin{lem}
If $x$ is compatible with a certificate $C$ then $x$ is aperiodic: 
\end{lem}
\begin{proof}
Fix $n\in\mathbb{N}$; it suffices to show that $x$ is $n$-aperiodic.
Let $F$ be an $n$-frame in $C$. Then $x$ is compatible with $F$,
so $CBC(x,F)$ is unorientable and, in particular, there is a bucket
$(n,\theta,\sigma)$ in $CBC(x,F)$ containing coins of different
orientations. If $w,w'\in F$ are witnesses  corresponding to these
coins, then the fact that they are in the same bucket means that $w=w'\bmod n$,
$\theta(w)=\theta(w')=\theta$ and $\sigma(w)=\sigma(w')=\sigma$.
The fact that the coins have different orientations means that $\pi_{\theta,\sigma}(x_{w})\neq\pi_{\theta,\sigma}(x_{w'})$.
Thus $x_{w}\neq x_{w'}$, and $x$ not $n$-periodic. 
\end{proof}
\begin{lem}
$X$ is shift invariant.
\end{lem}
\begin{proof}
If $x$ is compatible with a certificate $C$, then every shift of
$x$ is compatible with the corresponding shift of $C$, defined in
the obvious way, so the shifts of $x$ are also in $X$.
\end{proof}
Next, there is an obvious local way to define the distance between
$n$-boxes, $n$-frames and certificates, which makes the space of
certificates compact and metrizable. We would like the following to
be true:
\begin{lem}
$X$ is closed.
\end{lem}
The proof should go as follows: Suppose $x_{n}\rightarrow x$ in $A^{\mathbb{Z}^{2}}$
and $x_{n}$ in compatible with a certificate $C_{n}$. By compactness
of the space of certificates, we can pass to a subsequence and assume
that $C_{n}\rightarrow C$. Then $x$ is compatible with $C$, so
$x\in X$.

This argument is flawed because pointwise convergence of witnesses,
as points in $\mathbb{R}^{2}$, does not imply convergence of their
integer parts. For this reason, compatibility does not pass to the
limit.

One can correct this flaw as follows: rather than defining witnesses
to be points in $\mathbb{R}^{2}$, define them as pairs $(w,s)$ where
$w\in\mathbb{R}^{2}$ and $s\in\mathbb{Z}^{2}$ with $\left\Vert w-s\right\Vert _{\infty}\leq1/2$,
and use $s$ in place of $[w]$ in the definition of compatibility.
With this change, the space of certificates is still compact, and
convergence of witnesses includes convergence of the ``integer parts''
$s$, so compatibility does pass to the limit. and the proof above
goes through.

We continue to write $[w]$ instead of $s$ but this should not cause
any confusion.

\begin{lem}
$X\neq\emptyset$.
\end{lem}
\begin{proof}
Fix a dense certificate $C$ (it is not hard to see that, taking all
$h_{n},w_{n}$ large relative to their predecessors, dense certificates
exist). 

By Corollary \ref{cor:witnesses-differ-in-position-direction-or-section},
witnesses $w,w'\in C$ with the same orientation and section must
satisfy $d(w,w')>w_{1}>2$ and hence the integer points derived from
$w,w'$ satisfy $[w]\neq[w']$.

It follows that if for every frame $F\in C$ we fix a coin and bucket
configuration $c_{F}$ on $n^{2}\cdot|\Theta|\cdot|\Sigma|$ buckets,
then there is a configuration $x\in A^{\mathbb{Z}^{2}}$ with $CBC(x,F)=c_{F}$;
for, by the above, the symbols $(x_{w})_{(\theta(w),\sigma(w))}\in\{H,T\}$
can be defined independently as $w$ ranges over all witnesses in
$C$.

By choice of $N_{n}$ we can choose $c_{F}$ to be unorientable for
each frame $F\in C$. Then the configuration $x$ above is compatible
with $C$, and hence $x\in X$. This shows that the subshift $X$
is non-empty.
\end{proof}
Combining the lemmas above we get:
\begin{cor}
$X$ is a non-empty subshift.
\end{cor}

\section{\label{sec:Strong-irriducibilty}Strong irreducibility}

In this section we show that the subshift $X$ defined above is strongly
irreducible.

\subsection{Reduction to bounded, simply connected regions}

Throughout this section and sometimes later, we view $\mathbb{Z}^{2}$
as the vertex set of a graph with edges between vertices at distance
one (so $(u,v)\in\mathbb{Z}^{2}$ is connected to the vertices $(u\pm1,v)$,
$(u,v\pm1)$), and identify a subset of $\mathbb{Z}^{2}$ with the
induced subgraph. Thus we will speak of connected or disconnected
subsets of $\mathbb{Z}^{2}$, of connected components in the set,
etc.

Strong irreducibility with gap $g$ is the property
\begin{quote}
(A) $X$ admits gluing along all pairs $E,F\subseteq\mathbb{Z}^{2}$
satisfying $d(E,F)>g$.
\end{quote}
We now perform a sequence of reductions which show that in order to
establish strong irreducibility with gap $g$, it suffices to show
that the subshift admits gluing for $(\frac{1}{2}g-2)$-separated
sets that satisfy some additional properties. 

\subsubsection*{We can assume $F$ is finite and $E$ has finitely many connected
components}

Indeed, we first claim that (A) follows from
\begin{quote}
(B) $X$ admits gluing along all pairs $E,F\subseteq\mathbb{Z}^{2}$
with $d(E,F)>g$ and $E,F$ are finite (and in particular, have finitely
many connected components).
\end{quote}
For, given arbitrary $E,F$ with $d(E,F)\geq g$, and given $x,y\in X$,
for each $n$ we can use (B) to find a gluing $z_{n}$$\in X$ of
$x|_{E\cap[-n,n]^{2}}$ and $y|_{F\cap[-n,n]^{2}}$. Any accumulation
point of $(z_{n})$ is then a gluing of $x|_{E}$ and $y|_{F}$. 

\subsubsection*{We can assume connected components are $\frac{1}{2}g-2$ separated}

Specifically, we claim that (B) follows from 
\begin{quote}
(C) $X$ admits gluing along pairs $E,F$ that satisfy 
\begin{enumerate}
\item [C1.] $d(E,F)>\frac{1}{2}g-2$.
\item [C2.] $E,F$ have are finite.
\item [C3.] Every pair of connected components in $E$ or $F$ is $(\frac{1}{2}g-2)$-separated, 
\end{enumerate}
\end{quote}
Indeed, suppose that $E,F$ satisfy (B). We will enlarge and merge
pairs of components to obtain sets $E',F'$ containing $E,F$, respectively,
in a way that (C1)-(C3) are satisfied; then (C) ensures gluing along
$E',F'$, which implies gluing along $E,F$.

To this end, view $E,F$ as subsets of $\mathbb{R}^{2}$, and enlarge
$E$ by adding to it all line segments between pairs $u,v\in E$ that
are neighbors in $\mathbb{Z}^{2}$. Do the same for $F$. Note that
subsets of $E\cap\mathbb{Z}^{2}$ that were previously connected components
in the induced graph from $\mathbb{Z}^{2}$ correspond to components
in the topological sense, and that $E,F$ are still $g$-separated
after adding these line segments.

Given any $u,v\in E\cap\mathbb{Z}^{2}$ that lie in different connected
components of $E$ and satisfy $d(u,v)<\frac{1}{2}g$, introduce a
line segment $\ell=\ell_{u,v}$ connecting $u,v$. Let $\widehat{E}$
denote the union of $E$ and all line segments $\ell_{u,v}$ as above.
Define $\widehat{F}$ analogously.
\begin{claim}
Suppose that $\ell_{u,v}$ and $\ell_{u',v'}$ are two line segments
with $d(u,v)<\frac{1}{4}g$ and $d(u',v')<\frac{1}{4}g$. Then
\begin{enumerate}
\item If $w\in\mathbb{R}^{d}$ and $d(\ell_{u,v},w)<\frac{1}{4}g$ then
$d(\{u,v\},w)<\frac{1}{2}g$.
\item If $d(\ell_{u,v},\ell_{u',v'})<\frac{1}{4}g$ then $d(\{u,v\},\{u',v'\})<\frac{1}{4}g$.
\end{enumerate}
\end{claim}
\begin{proof}
(1) Since the length of $\ell_{u,v}$ is less than $\frac{1}{2}g$,
every $z\in\ell_{u,v}$ is at distance of no more than $\frac{1}{4}g$
from one of the endpoints. Therefore, if $d(w,\ell)\leq\frac{1}{4}g$,
and choosing $z\in\ell_{u,v}$ to be the point nearest $w$, we have
\[
d(w,\{u,v\})\leq d(w,z)+d(z,\{u,v\})<\frac{1}{4}g+\frac{1}{4}g=\frac{1}{2}g
\]
(2) If $d(\ell_{u,v},\ell_{u',v'})<\frac{1}{4}g$, then the distance
is attained a pair of points $z\in\ell_{u,v}$ and $z'\in\ell_{u',v'}$
with either $z\in\{u,v\}$ or $z'\in\{u',v'\}$. In the first case,
apply (1) with $w=z$ to conclude $d(z,\{u',v'\})<\frac{1}{2}g$,
and $z\in\{u,v\}$ gives us $d(\{u,v\},\{u',v'\})<\frac{1}{2}g$ as
claimed. The other case follows similarly.
\end{proof}
Since $\widehat{E},\widehat{F}$ differ from $E,F$ only by the addition
of segments $\ell_{u,v}$ as above between connected components, it
follows from the claim that if $I\subseteq\widehat{E}$ is a connected
component, then for any $u\in I$ we have $d(u,w)>\frac{1}{2}g$ for
$w\in\widehat{F}\cup(\widehat{E}\setminus I)$. 

We now return to $\mathbb{Z}^{2}$: For each line segment $\ell$
in $\widehat{E}$ add to $E$ the integer points in a $1$-neighborhood
of $\ell$. In the resulting set $E'$, each connected component (in
the graph $\mathbb{Z}^{2}$) is obtained from a connected component
of $\widehat{E}$, and the former lies in the $1$-neighborhood of
the latter. Do the same to obtain $F'$ from $\widehat{F}$. Now $E\subseteq E'$,
$F\subseteq F'$, and $E',F'$ satisfy (C1) and (C3); they also satisfy
(C2) because $E,F$ satisfy (C2).

\subsubsection*{We can assume that $E,F$ are connected }

Condition (C) above follows from
\begin{quote}
(D) $X$ admits gluing along pairs $E,F$ that satisfy 
\begin{enumerate}
\item [D1.] $d(E,F)>\frac{1}{2}g-2$.
\item [D2.] $E,F$ are finite and connected.
\end{enumerate}
\end{quote}
For suppose that $E,F$ satisfy the conditions (C1)-(C3) and assume
that (D) holds. We show that every $x,y\in X$ admit a gluing along
$E,F$. We do so by induction on the total number of connected components
in $E,F$. When there are two components altogether this is exactly
(D). Now suppose the implication is known when there are up to $n$
components, and assume $E\cup F$ have $n+1$ components. Consider
the partial order where components $J,J'$ satisfy $J<J'$ if $J'$
separates $J$ from infinity. Let $J$ be a connected component of
$E\cup F$ that is minimal with respect to this order and observe
that all other components of $E\cup F$ are contained in the unbounded
component $J'$ of $\{u\in\mathbb{R}^{d}\mid d(u,J)>\frac{1}{2}g-2\}$

Since $d(E,F)>0$, either $J\subseteq E$ or $J\subseteq F$. Suppose
for instance that $J\subseteq F$, the other case is analogous.

If $J=F$, then we simply use (D) to glue $x,y$ along $J',F$, and
this is also a gluing along $E,F$, as required.

If $J\neq F$ then $F'=F\cap J'$ has fewer connected components than
$F$ and $E,F'$ still satisfies (C) so by the induction hypothesis
we can glue $x,y$ along $E,F'$ to obtain $z$. Now we apply (D)
to glue $z$ and $y$ along $J',J$, yielding a gluing of $x,y$ along
$E,F$. 

\subsubsection*{One more (not strictly necessary) reduction }

Condition (D) follows from
\begin{quote}
(E) $X$ admits gluing along pairs $E,F$ that satisfy
\begin{enumerate}
\item $F$ is finite and simply connected.
\item $E$ is the unbounded component of $\{u\in\mathbb{Z}^{2}\mid d(u,F)>\frac{1}{2}g-2\}$.
\end{enumerate}
\end{quote}
This reduction is not ~necessary for the later proofs, but it gives
a simpler situation to think about. To see that (E) implies (D), suppose
that $E,F$ are as stated in (D). 

If $F$ does not separate $E$ from infinity, fill in all the holes
in $F$ to obtain a finite simply connected set $F'$, let $E'$ denote
the unbounded component of $\{u\in\mathbb{Z}^{2}\mid d(u,F')>\frac{1}{2}g-2\}$,
and note that $E\subseteq E'$. Therefore by (E) we can glue along
$E',F'$ and this gives a gluing along $E,F$.

If $F$ does separate $E$ from infinity, swap the roles of $E,F$,
and proceed as above. 

\subsubsection*{Summary}

We have reduced the problem of proving SI with gap $g$ to proving
that condition (E) above holds. This is what we do: we begin with
$g=10$ and must prove condition (E), noting that $\frac{1}{2}g-2=3$. 

\subsection{Proof setup}

Let $x,y\in X$ with compatible with dense certificates $C_{x},C_{y}$
respectively. 

Let $E_{x},E_{y}\subseteq\mathbb{Z}^{2}$ be connected sets with $E_{y}$
finite and $E_{x}=\{u\in\mathbb{Z}^{2}\mid d(u,E_{y})>3\}$, so $d(E_{x},E_{y})>3$.
We introduce the ``filled in'' versions of the sets
\begin{align*}
\widehat{E}_{x} & =E_{x}+(-1,1)^{2}\\
\widehat{E}_{y} & =E_{y}+(-1,1)^{2}
\end{align*}
and
\[
\widehat{E}=\mathbb{R}^{2}\setminus(\widehat{E}_{x}\cup\widehat{E}_{y})
\]
By our assumptions, $\widehat{E}$ separates $E_{x}$ from $E_{y}$
and is at least $1$-far from each of them. We call $\widehat{E}$
the \textbf{gap}. We call $\widehat{E}_{x}$ the \textbf{zone} of
$E_{x}$ and $\widehat{E}_{y}$ the \textbf{zone }of $E_{y}$. For
a frame $F$ in one of the certificates $C_{x},C_{y}$, we write $C_{F}=C_{x}$
or $C_{y}$ depending on whether $F\in C_{x}$ to $F\in C_{y}$, and
similarly $E_{F}=E_{x}$ or $E_{y}$ and $\widehat{E}_{F}=\widehat{E}_{x}$
or $\widehat{E}_{y}$. We write $C'_{F}$ for the certificate that
$F$ does not belong to and similarly $E'_{F},\widehat{E}'_{F}$.
We use similar notation for boxes and witnesses. 

Let
\[
z=x|_{E_{x}}\cup y|_{E_{y}}
\]
This is a partially defined configuration in $A^{\mathbb{Z}^{2}}$.
Our objective is to extend $z$ to a configuration in $X$. This means
that we must do two things:
\begin{enumerate}
\item We must define the symbols $z_{u}$ for $u\in(E_{x}\cup E_{y})^{c}$.
\item We must define a dense certificate $C$ compatible with $z$, thus
showing that $z\in X$. 
\end{enumerate}
Both the extension of $z$ and the definition of $C$ will be carried
out in an iterative fashion.  We begin from $C=\emptyset$, which
is trivially compatible with $z$. Frames will then be added to $C$
one at a time in several (infinite) rounds. Initially, we add frames
from $C_{x}$ or $C_{y}$ that require little or no modification in
order to maintain the certificate properties and be compatible with
$z$. Afterwards, we add frames from $C_{x},C_{y}$ that require more
extensive changes, and finally we add entirely new frames, although
they will also be derived, after substantial changes, from frames
in $C_{x},C_{y}$. 

When a frame $F$ is added to $C$, we (partially) define any undefined
symbols in $z$ at witnesses $w\in F$, setting the appropriate component
of $z_{w}\in\{H,T\}^{\Theta\times\Sigma}$ so as to ensure compatibility.
Because we only add frames and witnesses that maintain the certificate
properties, by Corollary \ref{cor:witnesses-differ-in-position-direction-or-section}
we never will attempt to define any component of a symbol more than
once. 

Once a frame is added to $C$ and corresponding symbols defined in
$z$, they are never changed at later stages of the construction.

\subsection{Performing the gluing }

Observe that $\widehat{E}_{y}$ is bounded, and set 
\[
n_{0}=\min\{n\in\mathbb{N}\mid\diam\widehat{E}_{y}<\frac{1}{10}r_{n}\}
\]

\subsection*{Step~A: Small frames far from the boundary}
\begin{quote}
For each $1\leq n<n_{0}$ in turn, for every $n$-frame $F\in C_{x}\cup C_{y}$
whose center $u=c(F)$ satisfies $d(u,\mathbb{R}^{2}\setminus\widehat{E}_{F})>\frac{2}{3}r_{n}$,
add $F$ to $C$, subject to the modifications below.
\end{quote}
Let $F$ be an $n$-frame as above. We define a new frame $F'$ (the
modified version of $F$) with the same center and the same boxes
as $F$, but possibly different witnesses. To make the change, we
consider each witness $w\in F$ in turn, and replace it with $w'$,
as follows:
\begin{itemize}
\item Let $B\in F$ be the box containing $w$, so $w$ is safe relative
to $\mathcal{R}_{C_{F}}(B)$. 
\item By Proposition \ref{prop:safe-paths}, there is a polygonal path $\gamma\subseteq\safe(\mathcal{R}_{C_{F}}(B))$
containing $w$, with orientation differing from $\theta(B)$ by at
most $\alpha_{1}=\frac{h_{1}}{w_{1}}$. We know that $\gamma$ is
an infinite path passing through $w\in B$, and it does not cross
the long edges of $B$ (because $\mathcal{R}_{C_{F}}(C)$ contains
rectangles whose sides are these same long edges). So, replacing $\gamma$
by $\gamma\cap B$, we can assume that $\gamma\subseteq B$ and that
it connects the short ends of $B$.
\item We define the point $w'$ differently in each of the following cases:
\begin{enumerate}
\item If $\gamma\cap\widehat{E}=\emptyset$, we set $w'=w$ (i.e. we do
not ``move'' $w$ at all).
\item Otherwise, choose $w'$ to be the point in $\gamma\cap\widehat{E}$
that is closest to the center of $F$. Such a point exists because
$\widehat{E}$ and $\gamma$ are compact. 
\end{enumerate}
\end{itemize}
Having added a $w'$ to each box $B\in F'$ we claim now that $C\cup\{F'\}$
is a certificate. Since the boxes in $F'$ come from those of $F$
they satisfy all the relevant conditions from the definition. It remains
to check two things:
\begin{itemize}
\item The center of $F'$ (equivalently, of $F$) is $r_{n}/2$-far from
the centers of other $n$-frames $G$ in $C$. There are two cases:

If $G\in C_{F}$, then $F,G$ are $n$-frames in the same certificate,
hence they are $r_{n}/2$ separated.

Otherwise, $G$ is from the other certificate than $F$. In this case,
each of the centers of $F,G$ lies each in its own zone, and each
is at least $\frac{2}{3}r_{n}$ from the complement of its zone (because
$F,G$ were added to $C$ in the present step, Step A, and this was
the condition for being added). Thus, the line segment connecting
$c(F),c(G)$ contains a segment of length $\frac{2}{3}r_{n}$ in each
of the two zones, hence its total length is greater than $\frac{4}{3}r_{n}$,
and certainly more than $\frac{1}{2}$$r_{n}$.
\item Every witness from a box $B'\in C\cup\{F'\}$ is safe with respect
to $\mathcal{R}_{C\cup\{F'\}}(B')$.

Since the addition of $F'$ to $C$ did not introduce any new boxes
at levels below $n$ we have not affected the safety of witnesses
from any level $\leq n$ in $C$, and $C$ does not yet contain any
frames at higher levels. So we only need to verify that the witnesses
$w'$ of $F'$ are safe.

Consider a witness $w'\in F'$ belonging to the section $\sigma=\sigma(w')$
of an $n$-box $B'\in F'$. Let $w,\gamma$ be as defined above when
$w'$ was added.

The family $\mathcal{R}=\mathcal{R}_{C\cup\{F'\}}(B')$ is derived
from sections of boxes added to $C$ from the certificate $C_{F}$
of $F$, and those added from the other certificate. 

Since all points in $\gamma$ are safe relative to $\mathcal{R}_{C_{F}}(B')$,
and since $w'\in\gamma$, we know that $w'$ is safe relative to $\mathcal{R}_{C_{F}}$. 

Thus, by part (1) of Proposition \ref{prop:safe-paths}, it suffices
to show that for $k\leq n$, if $S\in\mathcal{R}_{C\cup\{F'\}}(B')\setminus\mathcal{R}_{C_{F}}(B')$,
then $w'\notin20S$. 

The case $k=n$ follows directly from Lemma \ref{lem:separation-of-boxes-from-different-frames-in-same-level},
because the fact that $F',F''$ come from different certificates and
their centers belong to the corresponding zones and are at least $\frac{2}{3}r_{n}$
far from the complementary zone, means that their centers are more
than $\frac{1}{2}r_{n}$ apart, as required by the lemma.

Now suppose that $1\leq k<n$ and that there exists a level-$k$ box
$B''\in C\cup\{F'\}$ belonging to a frame $F''\in C\cup\{F'\}$,
and a section $S=(B'')^{(i)}$, such that $S\in\mathcal{R}_{C\cup\{F'\}}(B')\setminus\mathcal{R}_{C_{F}}(B')$
and $w'\in20S$. 
\begin{claim}
$d(w',\widehat{E}_{F''})<\frac{1}{6}r_{n}$
\end{claim}
\begin{proof}
[Proof of the Claim] Observe that
\begin{itemize}
\item Since $F''$ was added to $C$ in Step A, it intersects $\widehat{E}_{F''}$
(its own zone, which is different from that of $F'$). 

In particular, $d(S,\widehat{E}_{F''})\leq w_{k}$.
\item $S$ has dimensions $\frac{w_{k}}{1000}\times h_{k}$ and $h_{k}\ll w_{k}$.

In particular, since $w'\in20S$, we have $d(w',S)\leq w_{k}$.
\end{itemize}
Therefore, 
\[
d(w',\widehat{E}_{F''})\leq d(w',S)+d(S,\widehat{E}_{F''})\leq2w_{k}\ll r_{n}
\]
as claimed.

\end{proof}
\begin{claim}
When $\gamma$ is traversed starting from $w'$ towards the center
of $F$, we eventually enter $\widehat{E}_{F''}$.
\end{claim}
\begin{proof}
[Proof of the Claim] Since we have assumed that $d(c(F),\mathbb{R}^{2}\setminus\widehat{E}_{F})>\frac{2}{3}r_{n}$,
the previous claim implies that $d(w',c(F))>\frac{2}{3}r_{n}-\frac{1}{6}r_{n}=\frac{1}{2}r_{n}$.
Since $w_{n}=\frac{9}{10}r_{n}$ we conclude that the part of $\gamma$
extending from $w'$ towards $c(F)$ is at least $\frac{1}{4}r_{n}\gg2r_{k}$
long. Next, observe that
\begin{itemize}
\item Since $S\in\mathcal{R}_{C}(B'')$, we know that $\theta(B'')=\theta(B')$ 
\item When we travel along $\gamma$ from $w'$ towards the center of $F'$,
the slope of $\gamma$ relative to the direction $\theta(F')$ is
at most $h_{1}/w_{1}$, which may be assumed less than $1/1000$.
Thus, the direction of $\gamma$ deviates from $\theta(F')=\theta(F'')$
at most $1$ unit per $1000$. 
\item $B''$ is almost radial, so when the long edges are extended in direction
$\theta$ the ray passes within $\frac{1}{100}r_{k}$ of the center
of $F''$. 
\item $w'\in20S$, so the distance from $w'$ to the one of the lines extending
the long edges of $B''$ is at most $20h_{k}$.
\end{itemize}
Thus, we can travel along $\gamma$ from $w'$ towards $c(F)$, we
will pass within $\frac{1}{100}r_{k}+20h_{k}+\frac{1}{1000}r_{k}$
of the center of $F''$, and this distance can be assumed smaller
than $\frac{1}{10}r_{k}$. But every point within $\frac{1}{10}r_{k}$
of the center of $F''$ lies in the zone $\widehat{E}_{F''}$, because
if $F''$ was added to $C$ in step A. This proves the claim.
\end{proof}
To conclude, as we move along $\gamma$, we will certainly enter the
zone $\widehat{E}_{F''}$. Because $\widehat{E}$ is a connected set
separating the zones, this means that we also pass through $\widehat{E}$.
This shows that we are not in case (1) of Step A, so we are in case
(2). But the new point that we have found in $\gamma\cap\widehat{E}$
is closer to the center of $F'$ than $w'$ was, contradicting the
definition of $w'$ in (2).
\end{itemize}
We have shown that $C\cup\{F'\}$ is a certificate.

It remains to update $z$ at any witness $w'\in F'$ that differs
from the corresponding witness in $F$. Such witnesses lie in $\widehat{E}$.
We have already explained why in this case, $\pi_{(\theta(w),\sigma(w))}(z_{w'})$
is not yet specified. To determine its value, we imagine witnesses
moved one at a time and set $\pi_{(\theta(w),\sigma(w))}(z_{w'})$
to $H$ or $T$ according to the strategy in the coin-and-bucket game,
started from $CBC(x,F)$ if $F\in C_{x}$ or from $CBC(y,F)$ if $F\in C_{y}$. 

\subsection*{Step~B: Large frames}
\begin{quote}
For each $n\geq n_{0}$ in turn, add all $n$-frames $F\in C_{x}$
to $C$, subject to the modifications below.

(Note that we add only frames from $C_{x}$, not $C_{y}$. This is
where finiteness of $E_{y}$ plays a role).
\end{quote}
Let $F$ be an $n$-frame as above. We proceed as in Step A: we define
a new frame $F'$ (the modified version of $F$) with the same center
and the same boxes as $F$, and for each witness $w\in F$ we define
a witness $w'\in F'$ using the same procedure as in Step A, using
the same path $\gamma$ defined there. Having done this to all witnesses
we we add $F'$ to $C$.

We must verify that $C\cup\{F'\}$ is a certificate. We have added
frames at different levels than step A and all frames and boxes we
added come from $C_{x}$, so separation of boxes and frames is inherited
from $C_{x}$. We also can not have disrupted the safety property
of existing witnesses from Step A because we are adding only higher
levels than before. What we must check is that the new witnesses added
in the present stage satisfy the safety property.

So let $w'\in F'$ be a witness derived in this stage from a witness
$w\in F$ belonging to section $\sigma$ of a box $B\in F$. As explained
in Step A, $w'$ is safe with respect to $\mathcal{R}_{C_{x}}(B)$.
To show that it is safe relative to $\mathcal{R}_{C\cup\{F'\}}(B)$,
we must show that if $k<n$ and if $S\in\mathcal{R}_{C\cup\{F'\}}(B)\setminus\mathcal{R}_{C_{x}}(B)$
is the $i$-th section of a $k$-box $B''\in C_{y}$ for some $k<n$,
then $w'\notin20S$.

Suppose the contrary and let $k,B'',i,S$ is as above but $w'\in20S$.
We know that $B''\in C_{y}\cap C$, and this occurs only if $k<n_{0}$
and $B''$ belongs to a frame added during Step A.
\begin{claim}
$w'\notin\widehat{E}_{y}$.
\end{claim}
\begin{proof}
[Proof of the Claim] By our choice of $w'$, if $w'\notin\widehat{E}$
then $\gamma$ does not intersect $\widehat{E}$ at all. So, if $w'\in\widehat{E}_{y}$
then $w'\notin\widehat{E}$, and, since $\widehat{E}$ separates $\widehat{E}_{x},\widehat{E}_{y}$,
we must also have $\gamma\cap\widehat{E}_{x}=\emptyset$. It follows
that if $w'\in\widehat{E}_{y}$, then $\gamma\subseteq\widehat{E}_{y}$.
On the other hand $w'$ is an $n$-witness for $n\geq n_{0}$, so
the diameter of $\widehat{E}$ (and hence of $\widehat{E}\cup\widehat{E}_{y}$)
is less than $\frac{1}{10}r_{n}$, while the diameter of $\gamma$
is at least $w_{n}=\frac{9}{10}r_{n}$, so $\gamma\subseteq\widehat{E}_{y}$
is impossible, and consequently so is $w'\in\widehat{E}_{y}$.
\end{proof}
\begin{claim}
$|i-\sigma|\leq1$ (recall that $i-\sigma(S)$ and $\sigma=\sigma(w')$).
\end{claim}
\begin{proof}
[Proof of the Claim] Since $S\in\mathcal{R}_{C\cup\{F\}}(B)$, we
know that $S\cap B^{(i)}\neq\emptyset$ (where $i$ is the section
number of $S$). Thus, $w'\in B^{(\sigma)}$ and $d(w',S)<\frac{40}{1000}w_{k}$
(because $w'\in20S$) implies 
\[
d(B^{(\sigma)},(B^{(i)}))\leq d(w',S)+\diam S<\frac{40}{1000}w_{k}+(h_{k}+\frac{1}{1000}w_{k})\ll w_{n}
\]
Since non-adjacent sections of $B$ are at least $\frac{1}{1000}w_{n}$
apart, we may conclude that $|i-\sigma|\leq1$.
\end{proof}
Let $\gamma'$ denote the portion of $\gamma$ starting from $w'$
(but not including $w'$) and continuing as far as possible towards
the center $c(F)$ of $F$, ending in the short side of $B$ nearest
to $c(F)$. There we must consider two cases.
\begin{enumerate}
\item The length of $\gamma'$ is less than $2r_{k}$.

In this case, since $w_{n}\gg r_{k}$, the point $w'$ is already
close enough to the central short side of $B$, that $\sigma=1$.
By the claim above, we must have $i\in\{1,2\}$. Let $F''$ denote
the disk of $B''$. Then 
\[
d(S,c(F''))<\frac{1}{1000}w_{k}+\frac{1}{10}r_{k}<\frac{1}{3}r_{k}
\]
and since $d(w',S)<\frac{40}{1000}w_{k}$ (because $w'\in20S$), we
have 
\[
d(w',c(F''))<\frac{1}{2}r_{k}
\]
Since $F''$ was added in Step A, we know that the inner $1/2$ of
the disk $F''$ is contained in $\widehat{E}_{F''}=\widehat{E}_{y}$,
and we conclude that $w'\in\widehat{E}_{y}$. But by the previous
claim, this is impossible.
\item $\gamma'$ extends more than $2r_{k}$ towards the center of $B$.

Then, arguing as in Step A, there is a point on $\gamma'$ closer
than $r_{k}/10$ to $c(B'')$.

Since $B'$ was added in Step A, this means that $\gamma'\cap\widehat{E}_{y}\neq\emptyset$.

Therefore, since $\widehat{E}$ separates $\widehat{E}_{x}$ from
$\widehat{E}_{y}$, if $\gamma'\cap\widehat{E}_{x}\neq\emptyset$
then $\gamma'\cap\widehat{E}\neq\emptyset$.

But $\gamma'\cap\widehat{E}=\emptyset$, since any point in $\gamma'\cap\widehat{E}$
would be closer to $c(F)$ than $w'$, contradicting the definition
of $w'$.

We conclude that $\gamma'\subseteq\widehat{E}_{y}$, and consequently
(since $w'\in\overline{\gamma'}\subseteq\widehat{E}\cup\widehat{E}_{y}$
and $w'\notin\widehat{E}_{y}$) we have $w'\in\widehat{E}$.

Since $n\geq n_{0}$, it follows that $\widehat{E}_{y}$, and hence
also $\gamma'$, has diameter $<r_{n}/10$. 

Therefore the distance of $w'$ (one end of $\gamma'$) from the short
edge of $B$ that is closer to $c(F)$, is $<\frac{1}{10}r_{n}=\frac{1}{9}w_{n}$. 

It follows that $\sigma=\sigma(w')\leq\frac{1}{9}\cdot1000$ so $\sigma\leq111$.

Therefore, since $|i-\sigma|\leq1$, we have $i\leq112$.

On the other hand, $B''$ was added in Step A, so $d(c(F''),\widehat{E})>\frac{2}{3}r_{k}$,
and hence $d(c(F''),w')>\frac{2}{3}r_{k}$. Since $w'\in20S$ we have
$d(S,c(F''))>\frac{2}{3}r_{k}-20\diam(S)>\frac{1}{2}r_{k}$. This
clearly excludes $S$ belonging to the first $200$ sections of $B''$,
say (because these are all closer then $\frac{1}{2}r_{n}$ to $c(F'')$),
contradicting the property $i<112$.
\end{enumerate}
dvBoth cases lead to a contradiction, so $w'\notin20S$ and hence $w'$
is safe relative to $\mathcal{R}_{C\cup\{F'\}}(B)$. We have shown
that $C\cup\{F'\}$ is a certificate. 

We again define $z$ at the sites $w'$ above so as to restore compatibility,
this process is the same as in Step A.

\subsection*{Step~C: Small frames relatively near the boundary}
\begin{quote}
For each $1\leq n<n_{0}$ in turn, add to $C$ those $n$-frames $F\in C_{x}\cup C_{y}$
whose center $u$ satisfies $\frac{1}{4}r_{n}\leq d(u,\mathbb{R}^{2}\setminus\widehat{E}_{F})\leq\frac{2}{3}r_{n}$,
subject to the modifications below.

\end{quote}
What we keep from these frames is only their positions (i.e.~their
center), and completely redefine their boxes and witnesses. 

First, we observe that adding frames with these centers does not violate
the $r_{n}/2$ separation condition of $n$-frames. Indeed, since
$n\leq n_{0}$, at the end of Step C all $n$-frames in $C$ have
come from steps A and C of the construction (since Step B only adds
frames at levels higher than or equal to $n_{0}).$ Frames that came
from the same certificate are separated by assumption, whereas if
they came from different certificates then the centers are in their
own zones and $r_{n}/4$-far from $\widehat{E}$, hence they are $r_{n}/2$
from each other (see Step A item (1)).

It remains to explain how to define the boxes and witnesses of an
$n$-frame $F\in C_{x}\cup C_{y}$ added in Stage C.

Since $n<n_{0}$, the diameter of $\widehat{E}$ is at least $r_{n}/10$,
and by assumption $d(u,\mathbb{R}^{2}\setminus\widehat{E}_{F})\leq\frac{2}{3}r_{n}$.
This implies that $\widehat{E}\cap F$ has a connected component $\widehat{E}'$
with diameter at least $\frac{1}{10}r_{n}$. Furthermore, since $\frac{1}{4}r_{n}\leq d(u,\mathbb{R}^{2}\setminus\widehat{E}_{F})$,
we know that $\widehat{E}'\subseteq F\setminus\frac{1}{4}F$.

By Lemma \ref{lem:existence-of-many-almost-radial-rectangles}, we
can find $N_{n}+1$ rectangles $R_{i}$ that are $100h_{n}$-separated
and almost radial with respect to $F$, have a common orientation
$\theta\in\Theta$, and such that $\widehat{E}'$ intersects both
long edges of $R_{i}$.

Let $W_{>n}$ denote the set of witnesses $w\in C$ whose level $k$
is greater than $n$, and such that $w\in20S$ for some section $S$
of one of the rectangles $R_{i}$.

Then $W_{>n}$ contains at most one witness $w$, because $h_{n+1}\gg r_{n}$,
and by Lemma \ref{lem:separation-of-witnesses-with-identical-orientations-and-section},
the witnesses at levels $>n$ are $h_{n+1}$-separated, while all
the rectangles $R_{i}$ (and hence their sections) lie within a single
$r_{n}$-ball. 

Because the $R_{i}$ are $100h_{n}$-separated, it follows that for
all but at most one of them, all sections $S$ of $R_{i}$ satisfy
$w\notin20S$. If there is an exceptional one, throw it out.

We are left with at least $N_{n}$ ``good'' rectangles which we
denote $B_{1},\ldots,B_{N_{n}}$; these are the new boxes in $F$.

Apply Proposition \ref{prop:safe-paths} to each $B_{i}$ to find
a path $\gamma$ connecting its short edges and contained in the safe
points of $\mathcal{R}_{C}(B_{i})$. Since $\widehat{E}'$ connects
the long edges of $B_{i}$ we can choose $w\in\gamma\cap\widehat{E}'$;
this is the new witness of $B_{i}$. 

To conclude this step, we note that all witnesses $w$ in $F$ lie
in $\widehat{E}$ so we can define $\pi_{(\theta(w),\sigma(w))}(z_{w})$
to get any coin-and-bucket configuration we could want. We choose
an unorientable one.

\subsection*{Interlude: How dense are the frames so far?}

Taking stock, for $n\geq n_{0}$ the centers of $n$-frames in $C$
are the same as the centers of $n$-frames in $C_{x}$, so they are
$10r_{n}$ dense. But for $n<n_{0}$, although we have added to $C$
many frames from $C_{x},C_{y}$, we have not added those whose centers
come too close to the opposite zones. Therefore, for $1\leq n<n_{0}$,
the centers of $n$-frames in $C$ may not be $10r_{n}$-dense. 

We cannot directly add the missing frames from levels $\leq n_{0}$
because, being close to the opposite zone, they may come too close
to frames from the opposite zone that are already in $C$.

However, if there are frames close to the boundary, then by moving
them slightly away from the boundary (possibly several times, in different
directions), we can restore both separation and density. This is the
final stage of the construction.

\subsection*{Step~D: Adding frames near the boundary}
\begin{quote}
For each $1\leq n<n_{0}$, iterate over all sites $v_{1},v_{2},\ldots\in\mathbb{Z}^{d}$.
For each $v_{i}$ in turn, if it is not within $10r_{n}$ of the center
of some $n$-frame in $C$, we will add a new $n$-frame to $C$.
Its location is obtained by selecting a frame in $C_{x}\cup C_{y}\setminus C$
that is within $10r_{n}$ of $v_{i}$ and shifting it towards $v_{i}$,
as described below.
\end{quote}
Suppose $n$ is given and $v\in\mathbb{Z}^{d}$ is not within $10r_{n}$
of the center of any $n$-frame in $C$. Since $C_{x},C_{y}$ are
certificates, we know that there are frames $F_{x}\in C_{x}$ and
$F_{y}\in C_{y}$ whose centers are within $10r_{n}$ of $v$. Set
$F=F_{y}$ if $v\in E_{y}$ and $F=F_{x}$ otherwise. We deal here
with the latter case, the former is dealt with similarly.

Write $u=c(F)$ for the center of $F$. Then $d(u,\widehat{E})<\frac{1}{4}r_{n}$.
Indeed, we have assumed that $n<n_{0}$, and $d(u,\widehat{E})\leq\frac{2}{3}r_{n}$
because $F$ was not added to $C$ in Step A; these two facts, and
the fact that it was not added in Step C, imply that $d(u,\widehat{E})<\frac{1}{4}r_{n}$.

Let $e\in\widehat{E}$ be a point with $d(e,u)\leq\frac{1}{4}r_{n}$. 

Let $\ell$ denote the ray starting at $u$ and passing through $v$,
and let $u'\in\ell$ be the point closest to $u$ that satisfies $d(e,u')=\frac{3}{4}r_{n}$.
Note that $d(u',u)\geq\frac{1}{2}r_{n}$ because after moving a distance
$t<\frac{1}{2}r_{n}$ along $\ell,$we are at a point that is less
than $\frac{1}{4}r_{n}+t\leq\frac{1}{4}r_{n}+\frac{1}{2}r_{n}<\frac{3}{4}r_{n}$
from $e$.

We define a new $n$-frame $F'$ centered $u'$. Its boxes and witnesses
are defined similarly to the way they were defined in Step C, noting
that since $d(e,u')=\frac{3}{4}r_{n}$ and $\diam\widehat{E}\geq\frac{1}{10}r_{n}$
we are in a position to apply Lemma \ref{lem:existence-of-many-almost-radial-rectangles}
again. Also, note that $u'=c(F')$ is not $r_{n}/2$-close to the
center of any previously added $n$-frame in $C$, because if $F''$
were such a frame, then its center $u''$ would satisfy 
\[
d(u'',v)\leq d(u'',u')+d(u',v)\leq\frac{1}{2}r_{n}+(d(u,v)-d(u',u))<\frac{1}{2}r_{n}+(10r_{n}-\frac{1}{2}r_{n})=10r_{n}
\]
Then $u''=c(F'')$ would already be $10r_{n}$-close to $v$, contrary
to assumption.

All that is left now is to add $F'$ to $C$, and extend the definition
of $z$ to be compatible with it.

This concludes the construction.

\section{\label{sec:Concluding-remarks}Concluding remarks and problems}

We end with some questions that arise naturally from this work. Some
were already mentioned in the introduction. 
\begin{problem}
Do there exist strongly irreducible $\mathbb{Z}^{2}$ subshifts on
which $\mathbb{Z}^{2}$ acts freely?
\end{problem}
Indeed, one can show that our example contains points with one direction
of periodicity, so it does not provide an answer to the problem above.
Ronnie Pavlov \cite{Pavlov2023} has an example of a free square mixing
system, but this is still far from strong irreducibility.

One mechanism related to periodicity an freeness of the action is
the complexity of individual points in the system. In this vein we
can ask:
\begin{problem}
Given $c,\alpha>0$, is there a SI $\mathbb{Z}^{d}$ subshift $X$
such that 
\[
\forall x\in X\qquad N_{n}(x)\geq c2^{n^{\alpha}}
\]
or even
\[
\forall x\in X\qquad N_{n}(x)\geq c\cdot n^{\alpha}
\]
where $N_{n}(x)$ is the number of $n\times n$ patterns in $x$ ?
\end{problem}
The existence of a system satisfying the first condition with $\alpha>1$
would give a positive solution to the previous problem. At the other
extreme, a system that fails the second condition for $\alpha=2$
and small enough $c$ must contain periodic points, by results related
to Nivat's conjecture.

Another natural question is:
\begin{problem}
Let $d\geq3$. If $X$ is a SI $\mathbb{Z}^{d}$ SFT and $Z$ is an
infinite free minimal $\mathbb{Z}^{d}$-subshift and $h_{top}(Z)<h_{top}(X)$,
is there an embedding (an injective factor map) $Z\rightarrow X$
?
\end{problem}
Lightwood proved this for $\mathbb{Z}^{2}$. For $d\geq3$, Bland
\cite{Bland2022} proved that an embedding exists provided there is
some factor map $Z\rightarrow X$.  Also note that if we drop the
assumption that $X$ is free, then a positive answer to the last problem
would imply a negative one to the first problem and some cases of
the second.

As noted in the introduction, a SI $\mathbb{Z}^{2}$ SFT contains
periodic points. But the class of sofic shifts (symbolic factors of
SFTs is much broader. So it is natural to ask,
\begin{problem}
Does there exist a SI $\mathbb{Z}^{2}$ sofic shift without periodic
points?
\end{problem}
The extending SFT need not be SI itself. Our construction is a natural
candidate, since one could try to build an SFT that encodes the certificates,
and the factor map to $X$ would erase them. Our intuition is that
this probably cannot be done with our example, but perhaps some modification
of it would work.

Finally we repeat a well known open problem:
\begin{problem}
Do SI $\mathbb{Z}^{3}$-SFTs contain periodic points?
\end{problem}
This was one of the questions we began with. As noted in the introduction,
Theorem \ref{thm:main} lends some weak moral support for a negative
answer, but the problem remains very much open.

\bibliographystyle{amsplain}


\begin{dajauthors}
\begin{authorinfo}[hochman]
  Department of Mathematics\\
  The Hebrew University of Jerusalem\\
  Jerusalem 91904, Israel\\ 
  michael.hochman\imageat{}mail\imagedot{}huji\imagedot{}ac\imagedot{}il
\end{authorinfo}
\end{dajauthors}

\end{document}